\documentclass[12pt]{amsart}
\counterwithin{equation}{section}

\allowdisplaybreaks

\usepackage{bbm}
\usepackage{amssymb}
\usepackage{amsmath}
\usepackage[all]{xy}
\usepackage{tikz}
\usepackage{adjustbox}
\usepackage{amsfonts}
\usepackage{mathrsfs}
\usepackage{latexsym}
\usepackage{graphicx}
\usepackage{enumerate}
\usepackage{amscd,amssymb,amsmath,amsbsy,amsthm}
	\definecolor{linkred}{rgb}{0.7,0.2,0.2}
	\definecolor{linkblue}{rgb}{0,0.2,0.6}
	\definecolor{linkgreen}{rgb}{0,0.6,0.2}
\usepackage[colorlinks,plainpages,
    linkcolor=linkblue,
    citecolor=linkgreen,
    urlcolor=linkred]{hyperref}

\usepackage{zref-clever}

\zcsetup{
	cap,                      
	nameinlink=false,         
	lastsep = {, and }        
}

\zcRefTypeSetup{assumption}{
	Name-sg = Assumption ,
	name-sg = assumption ,
	Name-pl = Assumptions ,
	name-pl = assumptions ,
}
\zcRefTypeSetup{cor}{
	Name-sg = Corollary ,
	name-sg = corollary ,
	Name-pl = Corollaries ,
	name-pl = corollaries ,
}
\zcRefTypeSetup{defin}{
	Name-sg = Definition ,
	name-sg = definition ,
	Name-pl = Definitions ,
	name-pl = definitions ,
}
\zcRefTypeSetup{enumi}{
	Name-sg = {} ,
	name-sg = {} ,
	Name-pl = {} ,
	name-pl = {} ,
}
\zcRefTypeSetup{equation}{
	Name-sg = {} ,
	name-sg = {} ,
	Name-pl = {} ,
	name-pl = {} ,
}
\zcRefTypeSetup{eg}{
	Name-sg = Example ,
	name-sg = example ,
	Name-pl = Examples ,
	name-pl = examples ,
}
\zcRefTypeSetup{lem}{
	Name-sg = Lemma ,
	name-sg = lemma ,
	Name-pl = Lemmas ,
	name-pl = lemmas ,
}
\zcRefTypeSetup{prop}{
	Name-sg = Proposition ,
	name-sg = proposition ,
	Name-pl = Propositions ,
	name-pl = propositions ,
}
\zcRefTypeSetup{rem}{
	Name-sg = Remark ,
	name-sg = remark ,
	Name-pl = Remarks ,
	name-pl = remarks ,
}
\zcRefTypeSetup{theo}{
	Name-sg = Theorem ,
	name-sg = theorem ,
	Name-pl = Theorems ,
	name-pl = theorems ,
}

\zcsetup{
	countertype = {
		enumi=item, enumii=item, enumiii=item, enumiv=item
	},
	counterresetby = {
		enumii=enumi, enumiii=enumii, enumiv=enumiii
	}
}
\zcRefTypeSetup{item}{
	Name-sg = {} ,
	name-sg = {} ,
	Name-pl = {} ,
	name-pl = {} ,
}

\AddToHook{env/assumption/begin}{\zcsetup{countertype={theo=assumption}}}
\AddToHook{env/cor/begin}{\zcsetup{countertype={theo=cor}}}
\AddToHook{env/defin/begin}{\zcsetup{countertype={theo=defin}}}
\AddToHook{env/eg/begin}{\zcsetup{countertype={theo=eg}}}
\AddToHook{env/lem/begin}{\zcsetup{countertype={theo=lem}}}
\AddToHook{env/prop/begin}{\zcsetup{countertype={theo=prop}}}
\AddToHook{env/rem/begin}{\zcsetup{countertype={theo=rem}}}

\newcommand{\cref}[1]{\zcref{#1}}
\newcommand{\Cref}[1]{\zcref[S]{#1}}

\newtheorem{theo}{Theorem}[section]

\newtheorem{prop}[theo]{Proposition}
\newtheorem{lem}[theo]{Lemma}
\newtheorem{cor}[theo]{Corollary}

\theoremstyle{defin}

\newtheorem{defin}[theo]{Definition}
\newtheorem{rem}[theo]{Remark}
\newtheorem{eg}[theo]{Example}

\def\calK{{\mathcal{K}}}
\def\calO{{\mathcal{O}}}
\def\calR{{\mathcal{R}}}
\def\bfN{{\mathbf{N}}}
\def\Gr{\operatorname{Gr}}
\def\bGr{\overline{\operatorname{Gr}}}
\def\pt{\mathsf{pt}}
\def\Lie{\mathsf{Lie}}
\def\gr{\operatorname{gr}}
\def\Eu{\operatorname{Eu}}
\def\Diff{\operatorname{\mathsf{Diff}}}
\def\Hom{\operatorname{Hom}}
\def\Ber{\operatorname{\sf Ber}}

\tikzset{anchorbase/.style={>=To,baseline={([yshift=-0.5ex]current bounding box.center)}}}

\begin{document}

\title[Shifted twisted Yangians and Coulomb branch]{Quivers with Involutions and Shifted Twisted Yangians via Coulomb Branches}

\author{Yaolong Shen}
\address[Yaolong Shen]{School of Mathematical Sciences, East China Normal University, Shanghai 200241, China}
\email{yaolongshen3555@gmail.com}

\author{Changjian Su}
\address[Changjian Su]{Yau Mathematical Sciences Center, Tsinghua University, Beijing, China}
\email{changjiansu@mail.tsinghua.edu.cn}

\author{Rui Xiong}
\address[Rui Xiong]{Department of Mathematics and Statistics, University of Ottawa, 150 Louis-Pasteur, Ottawa, ON, K1N 6N5, Canada}
\email{xiongrui@umich.edu}

\begin{abstract}
    To a quiver with involution, we study the Coulomb branch of the 3d $\mathcal{N} = 4$ involution-fixed part of the quiver gauge theory. We show that there is an algebra homomorphism from the corresponding shifted twisted Yangian to the quantized Coulomb branch algebra. This gives a new instance of 3D mirror symmetries.
\end{abstract}

\maketitle

\setcounter{tocdepth}{1}

\section{Introduction}

It was proved by Ariki \cite{Ari} that
the category of finite dimensional modules of affine Hecke algebras in type $A$ can be used to categorify the negative part $\mathbf{U}^-$ of the Drinfeld--Jimbo quantum group $\mathbf{U}$ of type $A$, which is a generalization of the Lascoux--Leclerc--Thibon conjecture \cite{LLT}. For the affine Hecke algebras of type $B$, there is a similar categorification, conjectured by Enomoto and Kashiwara \cite{EK07,EK08,EK082}, and proved later by Enomoto \cite{En09} for some special cases and by Varagnolo and Vasserot in full generality \cite{VV11}. In this setting, the negative part $\mathbf{U}^-$ is replaced by some representation of a new quantum algebra ${}^\tau\mathbf{B}$, which is roughly speaking a Cartan-folded variation of $\mathbf{U}$, and is determined by an involution $\tau$ on the Dynkin diagram. 
$$
\begin{tikzpicture}[scale=1.2,
                    arrow/.style={red, <->, thick}]
\foreach \i in {0,...,5}
    \node[circle, draw, inner sep=3pt] (a\i) at (\i,0) {};
    \foreach \i in {0,...,5}
    \node (b\i) at (\i,.05) {};
\draw[arrow] (b0) to[bend left=50] (b5);
\draw[arrow] (b1) to[bend left=40] (b4);
\draw[arrow] (b2) to[bend left=30] (b3);
\node[red] at (2.5,1.5){$\tau$};
\draw[-,ultra thick] (a0) -- (a1); 
\draw[-,ultra thick] (a1) -- (a2); 
\draw[-,ultra thick] (a2) -- (a3); 
\draw[-,ultra thick] (a3) -- (a4); 
\draw[-,ultra thick] (a4) -- (a5); 
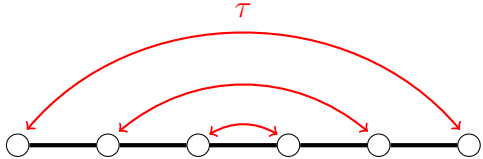
\node at (2.5,-.5){$\text{An example of a Dynkin diagram with an involution } \tau$};
\end{tikzpicture}
$$

\newcommand{\g}{\mathfrak{g}}

On the other hand, a \emph{Satake diagram} is a bi-colored Dynkin diagram 
\( I = I_\circ \cup I_\bullet \) equipped with a diagram involution \(\tau\), and it is called \emph{quasi-split} if \( I = I_\circ \). The Satake diagrams can be used to classify the real simple Lie algebras and the
symmetric pairs \((\mathfrak{g}, \mathfrak{g}^\theta)\), where \(\mathfrak{g}\) is a complex simple Lie algebra and \(\theta\) is an involution of \(\mathfrak{g}\), see \cite{Ara62}. Quantizing \((\mathfrak{g}, \mathfrak{g}^\theta)\), we get the quantum symmetric pair \((\mathbf{U},\mathbf{U}^\imath)\), where \(\mathbf{U}^\imath\) is a coideal subalgebra of \(\mathbf{U}\), known as the \(\imath\)quantum group \(\mathbf{U}^\imath\), see \cite{Let99, Kol14, Wan23}. 
Over the past decade, many fundamental constructions for quantum groups---including Jimbo--Schur duality, canonical bases, (quasi) \(R\)-matrices, and Hall algebra realizations---have been extended to the framework of \(\imath\)quantum groups; cf. \cite{BW18p, BW18, BK19, LW23,SW23,WZ22,SW24}.

The following is a comparison of the presentations of two algebras ${}^\tau\mathbf{B}$ and $\mathbf{U}^\imath$: 
$$
\begin{matrix}
\text{Enomoto--Kashiwara's algebra}\\
\begin{array}{|c|}\hline\vphantom{\dfrac{1}{2}}
{}^\tau\mathbf{B}=\left<T_i^{\pm 1},E_i,F_i\right>_{i\in I}
\\[1ex]
\begin{aligned}
& T_iT_j = T_jT_i,\qquad T_{\tau(i)}=T_i\\
& T_iE_jT_i^{-1} = q^{c_{ij}+c_{\tau i,j}} E_j\\
& T_iF_jT_i^{-1} = q^{-c_{ij}-c_{\tau i,j}} F_j\\
& E_iF_j - q^{-c_{ij}}F_jE_i=
(\delta_{ij}+\delta_{\tau i,j}T_i)\\
& \text{Serre relations for $E_i$'s and $F_i$'s}
\end{aligned}
\\\hline
\end{array}\\\vphantom{\dfrac{1}{2}}
\text{Higgs branch side}
\end{matrix}
\quad
\begin{matrix}
\stackrel{\mbox{\large\tt?}}\longleftrightarrow\\
\vphantom{\dfrac{1}{2}}
\end{matrix}
\quad
\begin{matrix}
\text{$\imath$Quantum groups}\\
\begin{array}{|c|}\hline\vphantom{\dfrac{1}{2}}
\mathbf{U}^\imath=\left<k_i^{\pm 1},B_i\right>_{i\in I}
\\[1ex]
\begin{aligned}
& k_ik_j = k_jk_i,\qquad k_{\tau i}=k_i^{-1}\\
& k_iB_jk_i^{-1} = q^{c_{ij}-c_{\tau i,j}} B_j\\
& B_iB_{j}- B_jB_i\stackrel{(c_{ij}=0)}{=\!\!=\!\!=\!\!=\!\!=}\delta_{\tau i,j}\frac{k_i-k_i^{-1}}{q-q^{-1}}\\
& \text{$\imath$Serre relations for $B_i$'s}
\end{aligned}
\\\hline
\end{array}\\\vphantom{\dfrac{1}{2}}
\text{Coulomb branch side}
\end{matrix}
$$
Since ${}^\tau\mathbf{B}$ and $\mathbf{U}^\imath$ can be defined from the same underlying combinatorial data, yet are in general quite different as algebras, it is natural to ask whether there is a deeper relationship between them.  In this paper, we show that such a connection arises through \emph{three-dimensional mirror symmetry}.  More precisely, we consider a gauge theory $(G, \mathbf{N})$ in which the algebra ${}^\tau\mathbf{B}$ naturally corresponds to the Higgs branch side, while the algebra $\mathbf{U}^\imath$ arises from the Coulomb branch side.

The Higgs/Coulomb branches originate from three-dimensional $\mathcal{N} = 4$ supersymmetric gauge theories, see \cite{Na16,BF19,WY23} for introductions aimed at mathematicians. Such a theory is expected to possess a well-defined moduli space of vacua, which contains two distinguished components known as the \emph{Higgs branch} and the \emph{Coulomb branch}.  
These branches form a mirror-symmetric dual pair.  
The Higgs branch admits a mathematical description as a Hamiltonian reduction, while a rigorous definition of Coulomb branches was only recently given by Braverman, Finkelberg, and Nakajima \cite{BFN18} via the geometry of the affine Grassmannian.

The gauge theory we use to establish the connection between ${}^\tau\mathbf{B}$ and $\mathbf{U}^\imath$ arises from the representation theory of a quiver with an involution $(Q,\tau)$. 
Firstly, we shall orient the Satake diagram to get a quiver with involution. For example,
$$
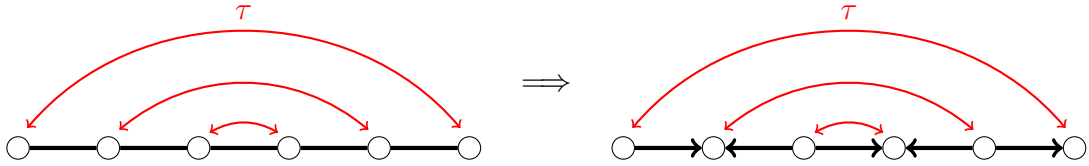
\begin{matrix}
\begin{tikzpicture}[scale=1.2,
                    arrow/.style={red, <->, thick}]
\foreach \i in {0,...,5}
    \node[circle, draw, inner sep=3pt] (a\i) at (\i,0) {};
    \foreach \i in {0,...,5}
    \node (b\i) at (\i,.1) {};
    \node[red] at (2.5,1.5){$\tau$};
\draw[arrow] (b0) to[bend left=50] (b5);
\draw[arrow] (b1) to[bend left=40] (b4);
\draw[arrow] (b2) to[bend left=30] (b3);
\draw[-,ultra thick] (a0) -- (a1); 
\draw[-,ultra thick] (a1) -- (a2); 
\draw[-,ultra thick] (a2) -- (a3); 
\draw[-,ultra thick] (a3) -- (a4); 
\draw[-,ultra thick] (a4) -- (a5); 
\end{tikzpicture}
\end{matrix}
\quad\Longrightarrow\quad
\begin{matrix}
\begin{tikzpicture}[scale=1.2,
                    arrow/.style={red, <->, thick}]
\node[red] at (2.5,1.5){$\tau$};
\foreach \i in {0,...,5}
    \node[circle, draw, inner sep=3pt] (a\i) at (\i,0) {};
    \foreach \i in {0,...,5}
    \node (b\i) at (\i,.1) {};
\draw[arrow] (b0) to[bend left=50] (b5);
\draw[arrow] (b1) to[bend left=40] (b4);
\draw[arrow] (b2) to[bend left=30] (b3);
\draw[->,ultra thick] (a0) -- (a1); 
\draw[<-,ultra thick] (a1) -- (a2); 
\draw[->,ultra thick] (a2) -- (a3); 
\draw[<-,ultra thick] (a3) -- (a4); 
\draw[->,ultra thick] (a4) -- (a5); 
\end{tikzpicture}
\end{matrix}
$$
Let $Q_0$ denote the set of nodes and $Q_1$ denote the set of edges of $Q$. Suppose $V$ and $W$ are two $Q_0$-graded vector spaces. 
We denote 
$$
G_V =\prod_{i\in Q_0}GL(V_i),\qquad 
E_V =\bigoplus_{h\in Q_1}\Hom(V_{s(h)},V_{t(h)}),\qquad 
L_{W,V}=\bigoplus_{i\in Q_0}\Hom(W_i,V_i).$$
Note that the pair $(G_V,E_V\oplus L_{W,V})$ defines a quiver gauge theory \cite{BFN19}. 

Now suppose that $V$ is equipped with a nondegenerate symmetric bilinear form 
\(
\langle - , - \rangle : V \times V \to \mathbb{C}
\)
that is compatible with the involution $\tau$ on the quiver, i.e., for each $i \in Q_0$, the orthogonal complement
\(
V_i^\perp = \bigoplus_{j \neq \tau i} V_j .
\)
This structure allows us to restrict both the gauge group and the representation space:
\[
G_V^\tau = \{ g \in G_V \mid g^t = g^{-1} \} \subseteq G_V, \qquad
E_V^\tau = \{ f \in E_V \mid f^t = -f \} \subseteq E_V ,
\]
where $g^t$ and $f^t$ denote transposes with respect to $\langle - , - \rangle$. 
Thus we only allow orthogonal gauge transformations and skew-symmetric linear maps, see \cite{En09,VV11}. The combined space
\(
E_V^\tau \oplus L_{W,V}
\)
parametrizes \emph{skew-symmetric $Q$-representations on $V$ framed by $W$}. For example, in the following diagram, where red arrows indicate pairings of edges under the involution $\tau$, whenever there is a matching between two linear maps $f\mathop{\color{red}\leftrightarrow} g$, we impose the condition
$f+g^t=0$. 
$$
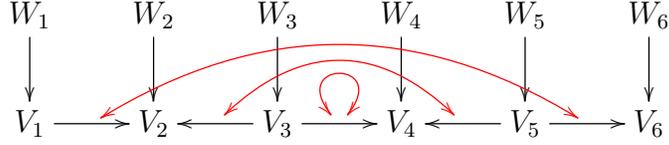
\begin{matrix}
\xymatrix{
W_1\ar[d]& W_2\ar[d]& W_3\ar[d]&
W_4\ar[d]& W_5\ar[d]& W_6\ar[d]\\
V_1\ar@{->}[r]_{}="f1" & V_2 \ar@{<-}[r]_{}="f2" & V_3\ar@{->}[r]_{}="f3" & V_4 \ar@{<-}[r]_{}="f4"& V_5\ar@{->}[r]_{}="f5" & V_6
\ar@[red]@{<->}@/^2.5pc/"f1";"f5"
\ar@[red]@{<->}@/^2pc/"f2";"f4"
\ar@[red]@(ul,ur)@{<->}"f3";"f3"}
\end{matrix}$$

We will consider the gauge theory $(G,\bfN)$ where 
$$
G=G_V^{\tau},\textit{\quad and \quad}
\bfN = E_V^\tau\oplus L_{W,V}. $$
We also make the following assumptions:
(1) the quiver $Q$ is simply-laced; (2) the involution $\tau$ has no fixed vertices. 
The assumption (2) also appeared in both \cite{En09} (right after Prop. 2.10) and \cite[Section 4.1]{VV11}. 

On one hand, by \cite{En09} and more generally by \cite{VV11}, 
there is an action of ${}^\tau\mathbf{B}$ on 
the category of $G$-equivariant perverse sheaves over $\bfN$, and it categorifies a representation ${}^\tau V(\lambda)$ of ${}^\tau\mathbf{B}$. 
This can be viewed as the study of the $\mathcal{N}=2$ Higgs branch associated with $(G,\bfN)$.

On the other hand, 
for simply-laced Dynkin quivers, there exists a surjective homomorphism from the shifted Yangian associated with the underlying simple Lie algebra onto the quantized Coulomb branch algebra, see \cite{BFN19}. In the context of symmetric pairs, the natural analogues are given by the twisted Yangians, which appear as coideal subalgebras of ordinary Yangians. Lu–Zhang \cite{LZ24} recently established a Drinfeld-type presentation for all quasi-split twisted Yangians, and showed that twisted Yangians can be obtained as degenerations of the Drinfeld-type presentations of affine $\imath$quantum groups, see also \cite{LW21,Zh22,LWZ25}. From a historical perspective, twisted Yangians arose in the study of reflection equations and boundary integrable models \cite{Ch84,Sk88}. They were first introduced by Olshanski \cite{Ol92} in the $R$-matrix formalism for types AI and AII, extended to type AIII by Molev–Ragoucy \cite{MR02}, and further developed for other classical types by Guay–Regelskis \cite{GR16}.

Utilizing the Drinfeld-type presentation of the twisted Yangians, we can define shifted twisted Yangians $\mathbf{Y}^\imath_\mu$ (see \cref{shiftedTY}) for any $\tau$-invariant coweight $\mu$. The following is our main result:

\begin{theo}[\cref{thm:main}]
There exists an algebra homomorphism from shifted twisted Yangian $\mathbf{Y}^\imath_\mu$ to the quantized Coulomb branch algebra $\mathcal{A}_\hbar$ of $(G,\bfN)$. 

\end{theo}

When the Satake diagram consists of two identical Dynkin diagrams, with the involution interchanging the nodes in corresponding positions, our result recovers the construction of \cite{BFN19}; see also \cite{We19}. 

The theorem is proved as follows. Via equivariant localization, we can embed $\mathcal{A}_\hbar$ into a suitable difference operator algebra. On the other hand, we produce a Gerasimov--Kharchev--Lebedev--Oblezin (GKLO)-type \cite{GKLO05}  homomorphism from $\mathbf{Y}^\imath_\mu$ to the same difference algebra, and we show that this homomorphism factors through the quantized Coulomb branch algebra $\mathcal{A}_\hbar$. 
By our construction, the Cartan generators $h_{i,r}$ of the shifted twisted Yangians are mapped into the Gelfand--Tsetlin subalgebra, while the generators $b_{i,s}$ are mapped to minuscule monopole operators.

Our work is influenced by and partially confirms a proposal of Lu, Wang, and Weekes \cite{LWW} that links shifted twisted Yangians to GKLO representations, affine Grassmannian slices, and Coulomb branches. Their work provides a comprehensive discussion of this framework, including a uniform construction of GKLO-type representations for shifted twisted Yangians of all quasi-split types.

Let us close the introduction by discussing some further directions. Firstly, in the untwisted setting, the homomorphism in \cref{thm:main} becomes surjective after a base change to the function field $\mathbb{C}(\hbar)$, see \cite{BFN19, We19}. We expect this to hold in the twisted setting as well. We note that the arguments in \cite{We19} do not generalize directly to our setting. Secondly, on the K-theory level, Finkelberg--Tsymbaliuk proved that the shifted quantum affine algebras map homomorphically into the K-theoretic quantized Coulomb branch of quiver gauge theories \cite{FT19}. The methods in this paper can be used to show that such an algebra homomorphism also exists for the shifted affine $\imath$quantum groups, and we will address this in a subsequent work.
Lastly, the Higgs and Coulomb branches are usually related by the three-dimensional mirror symmetry. For example, the Higgs branch of a quiver gauge theory is the Nakajima quiver variety, while the Coulomb branch is the generalized affine Grassmannian slice \cite{BFN19}. We expect that for the gauge theory $(G,\mathbf{N})$, the Higgs branch
might be a generalization of the Nakajima quiver variety.
Moreover, as observed in \cite{LWW}, the Coulomb branch should be related to the fixed-point loci of the affine Grassmannian slices.

\subsection*{Acknowledgment}
We would like to thank David Hernandez, Kang Lu, Peng Shan, Weiqiang Wang, Alex Weekes, Weinan Zhang and Yehao Zhou for useful discussions. We also thank the anonymous referee for helpful suggestions.
CS is supported by the National Key R\&D Program of China (No. 2025YFA1017400).
YS is partially supported by the Fields Institute.

\section{Coulomb branches of cotangent type}

\subsection{BFN Coulomb Branch}
Let $G$ be a complex reductive group with a Borel subgroup $B$ and a maximal torus $T$. Let $X_*(T)$ be the cocharacter lattice of $T$ with dominant ones denoted by $X_*(T)^+$. Recall the dominance order on $X_*(T)^+$: for any $\lambda,\mu\in X_*(T)^+$, $\lambda\leq \mu$ iff $\mu-\lambda$ is a nonnegative linear combination of the positive coroots.
Let
$\calO=\mathbb{C}[[z]]$, 
$\calK=\mathbb{C}(\!(z)\!)$. We write $G_\calK:=G(\calK)$, $G_\calO:=G(\calO)$ and $T_\calK:=T(\calK)$, respectively. Any cocharacter $\lambda:\mathbb{G}_m\rightarrow T$ gives a homomorphism $\mathcal{K}^*\rightarrow T_\mathcal{K}$, and we let $z^\lambda$ denote the image of $z\in \mathcal{K}^*$.
Recall the affine Grassmannian is defined to be $\Gr_G=G_\calK/G_\calO$. By the Cartan decomposition,
\[\Gr_G=\bigsqcup_{\lambda\in X_*(T)^+}\Gr_\lambda,\]
where $\Gr_\lambda:=G_\calO z^\lambda G_\calO/G_\calO$. Its closure $\bGr_\lambda=\bigsqcup\limits_{\mu\in X_*(T)^+,\mu\leq \lambda}\Gr_\mu$.

Let $\bfN$ be a complex representation of $G$. Let $\bfN_\calK:=\bfN(\calK)$ and $\bfN_\calO:=\bfN(\calO)$, respectively.
We consider
$$\mathcal{T}_{G,\bfN}
=G_\calK\times_{G_\calO} \bfN_\calO
=\{(gG_\calO,x):x\in g\bfN_\calO\}\subset 
\Gr_G\times \bfN_\calK.$$
The second projection $\mathcal{T}_{G,\bfN}\to \bfN_\calK$ could be viewed as an analogue of the Springer resolution. 
The (classical) \emph{BFN space} $\calR_{G,\bfN}$ \cite{BFN18} is defined by the pullback square
$$\xymatrix{
\calR_{G,\bfN} \ar@{^{(}->}[r] \ar@{^{(}->}[d] & \Gr_G\times \bfN_\calO \ar@{^{(}->}[d]\\
\mathcal{T}_{G,\bfN} \ar@{^{(}->}[r]&\Gr_G\times \bfN_\calK.}$$
Written explicitly, we have
$$\calR_{G,\bfN} = \{(gG_\calO,x): x\in g\bfN_\calO\cap \bfN_\calO\}\subset \Gr_G\times \bfN_\calO.$$
The definition is motivated by writing the Steinberg type variety in the following form:
\begin{equation}\label{eq:TNT=GGR}
\mathcal{T}_{G,\bfN}\times_{\bfN_\calK}
\mathcal{T}_{G,\bfN}
=G_\calK\times_{G_\calO}\calR_{G,\bfN}.
\end{equation}
Following \cite[Section 2(i)]{BFN18}, we consider the $\mathbb{C}^\times$ action on $\calR_{G,\bfN}$, which rotates $z\in \calO$ by weight $1$ and scales $\bfN$ by weight $\frac{1}{2}$ simultaneously. 
Let $\hbar$ be the equivariant parameter of this $\mathbb{C}^\times$, and denote this $\mathbb{C}^\times$ by $\mathbb{C}^\times_\hbar$. Hence,  
$H_{\mathbb{C}^\times_\hbar}^*(\pt)=\mathbb{C}[\hbar]$. 
The \emph{(quantized) Coulomb branch algebra} is defined to be the Borel--Moore homology
$$\mathcal{A}_\hbar(G,\bfN) = 
H^{G_{\calO}\rtimes\mathbb{C}_\hbar^\times}_*(\calR_{G,\bfN}).$$
By \cite{BFN18}, $\mathcal{A}_\hbar(G,\bfN)$ admits a convolution product induced by \eqref{eq:TNT=GGR}. 
The subalgebra $H^{G_\calO\rtimes \mathbb{C}_\hbar^\times}(\pt)\subset \mathcal{A}_\hbar(G,\bfN)$ is called the \emph{Gelfand--Tsetlin subalgebra}. 

Now let us discuss some background for the computation of the Coulomb branch algebra $\mathcal{A}_\hbar(G,\bfN)$. 
Recall $T$ is a maximal torus of $G$. Let $W$ be the Weyl group, and let $\mathfrak{t}$ be the Lie algebra of $T$, and $\mathbb{A}^1$ be the Lie algebra of $\mathbb{C}^\times_\hbar$.

First of all, the Coulomb branch algebra can be embedded into the difference algebra of $T$ as follows. 
We define the \emph{difference algebra} $\Diff_\hbar(T)$ of $T$ to be 
$$\Diff_\hbar(T)=\mathbb{C}(\mathfrak{t}\times \mathbb{A}^1)\rtimes X_*(T),$$
where $\mathbb{C}(\mathfrak{t}\times \mathbb{A}^1)$ is the field of rational functions over $\mathfrak{t}\times \mathbb{A}^1$. 
More precisely, $\Diff_\hbar(T)$ is the free $\mathbb{C}(\mathfrak{t}\times \mathbb{A}^1)$-module with basis $d_\mu$ for cocharacters $\mu\in X_*(T)$ such that 
$$(f(t,\hbar)d_\lambda)\cdot (g(t,\hbar)d_\mu)=f(t,\hbar)g(t+\lambda\hbar,\hbar)d_{\lambda+\mu}.$$
We could view elements in $\Diff_\hbar(T)$ as difference operators on $\mathbb{C}(\mathfrak{t}\times \mathbb{A}^1)$. 
By \cite[Lemma 5.10]{BFN18}, we have an embedding
$
\iota_*:\mathcal{A}_\hbar(T,\bfN)^W
\hookrightarrow
\mathcal{A}_\hbar(G,\bfN)$.
By \cite[Lemma 5.9]{BFN18}, it is an isomorphism after base changing to the fraction field of $H_T^*(\pt)$. 
By \cite[Lemma 5.11]{BFN18}, we have the following embedding
$
z^*:\mathcal{A}_\hbar(T,\bfN)
\hookrightarrow
\mathcal{A}_\hbar(T,0)$. 
The explicit computation in \cite[Section 4]{BFN18} implies $\mathcal{A}_\hbar(T,0)\hookrightarrow \Diff_\hbar(T)$ with 
$H_{T_\calO\times \mathbb{C}_\hbar^\times}^*(\pt)$ identified with the $\mathbb{C}[\mathfrak{t}\times \mathbb{A}^1]$. 
As a result, the composition of the above maps gives 
an embedding of algebras
\begin{equation}\label{eq:embeddiff}
z^*\circ (\iota_*)^{-1}:\mathcal{A}_\hbar(G,\bfN)\stackrel{\subset}\longrightarrow \Diff_\hbar(T). 
\end{equation}
In this paper, we will identify 
$\mathcal{A}_\hbar(G,\bfN)$ as a subalgebra of the difference algebra $\Diff_\hbar(T)$. 

Secondly, the Coulomb branch algebra is equipped with a structure of filtered algebra, whose associated graded algebra can be described explicitly as follows. 
Recall there is a natural projection 
$$\pi:\calR_{G,\bfN}\longrightarrow \Gr_G.$$
Let us denote 
$$\calR_{\lambda}=\pi^{-1}(\Gr_\lambda),\qquad \textit{ and } \calR_{\leq \lambda} = \pi^{-1}(\bGr_\lambda).$$
By \cite[Proposition 6.1]{BFN18}, the Coulomb branch algebra $\mathcal{A}_\hbar(G,\bfN)$ is naturally filtered by 
$H^{G_\calO\times\mathbb{C}_\hbar^\times}_*(\calR_{\leq \lambda})$ and the associated graded algebra is 
$$\gr\mathcal{A}_\hbar(G,\bfN)=\bigoplus_{\lambda\text{ dominant}} H_*^{G_\calO\times\mathbb{C}_\hbar^\times}(\calR_\lambda).$$
Note that we can identify
$
H_*^{G_\calO}(\calR_\lambda) \simeq H_*^{G_\calO}(\Gr_\lambda) \simeq \mathbb{C}[\mathfrak{t}]^{W_\lambda}$ where $W_\lambda$ is the stabilizer group of $\lambda$ in the Weyl group $W$. 
For $f\in \mathbb{C}[\mathfrak{t}]^{W_\lambda}\simeq H_*^{G_\calO}(\calR_\lambda)$, we define the corresponding element by 
$$f[\calR_\lambda]\in \gr\mathcal{A}_\hbar(G,\bfN).$$
From the above construction, a lifting of $f[\mathcal{R}_\lambda]$ takes the following form, embedded into the difference algebra \eqref{eq:embeddiff} (see \cite{BFN18}),
\begin{equation}\label{eq:monopoleop}
\sum_{w\in W^\lambda}
w\left(f\cdot \frac{\Eu\left(
z^\lambda\bfN_\calO/(z^\lambda\bfN_\calO\cap\bfN_\calO)\right)}{
\Eu(T_{\lambda}\Gr_\lambda)
}\right)
d_{w\lambda}+\text{lower term},
\end{equation}
where $W^\lambda$ is the set of minimal length representatives of cosets in $W/W_\lambda$,  $\Eu(-)$ is the $T\times\mathbb
{C}^\times_\hbar$-equivariant Euler class, i.e. the product of weights, and $T_{\lambda}\Gr_\lambda
$ is the tangent space of $\Gr_\lambda$ at the torus fixed point $z^\lambda G_\calO/G_\calO$. Note that we have a bijection $W^\lambda\cong W\cdot \lambda$ by $w\mapsto w\lambda$. 
Here ``lower term'' means linear combinations of $d_\mu$ with $\mu^+< \lambda$ under the dominant order, where $\mu^+$ is the unique dominant cocharacter in its Weyl group orbit $W\mu$. 
Note that $\Eu(T_\lambda\Gr_\lambda)$ is the product of inversions of the (extended) affine Weyl group element $t_\lambda$ of translation by $\lambda$. More generally, for any $\mu\in W\lambda$, 
\begin{equation}\label{equ:wts}
    T_\mu\Gr_\lambda = \bigoplus_{\alpha\in R, \langle\mu,\alpha\rangle\geq 1}\bigoplus_{n=0}^{\langle\mu,\alpha\rangle - 1}\mathfrak{g}_\alpha z^n,
\end{equation}
where $R$ is the set of roots of $G$.

Finally, we need the minuscule monopole operators.
Recall a cocharacter $\lambda$ is called \emph{minuscule} if 
$\langle \lambda,\alpha\rangle\in \{0,1\}$
for all positive roots $\alpha$. Equivalently, it is a minimal dominant cocharacter under the dominant order. 
As a result, when $\lambda$ is minuscule, \eqref{eq:monopoleop} has no lower terms, i.e.
\begin{equation}\label{eq:minu-mono-op}
f[\mathcal{R}_\lambda]=\sum_{w\in W^\lambda}
w\left(f\cdot \frac{\Eu\left(
z^\lambda\bfN_\calO/(z^\lambda\bfN_\calO\cap\bfN_\calO)\right)}{
\Eu(T_{\lambda}\Gr_\lambda)
}\right)
d_{w\lambda}\in \mathcal{A}_\hbar(G,\bfN).
\end{equation}
On the other hand, we have 
$\Gr_\lambda=\bGr_\lambda=G/P_\lambda$ where $P_\lambda$ is a parabolic subgroup whose Weyl group is $W_\lambda$. 
In particular, 
$T_\lambda\Gr_\lambda=T_{1}G/P_\lambda=\mathfrak{g}/\mathfrak{p}_\lambda$, where $\mathfrak{p}_\lambda=\operatorname{Lie}P_\lambda$. 

It is worth mentioning the classification of minuscule cocharacters. 
By definition, a cocharacter of $G$ is minuscule if and only if its image in $G_{ad}=G/Z(G)$ is minuscule. 
A cocharacter $\lambda$ of $G_1\times \cdots \times G_\ell$ is minuscule if and only if the component $\lambda_i$ is minuscule for each $1\leq i\leq \ell$. 
For a simple group $G$, a nontrivial minuscule cocharacter is necessarily a fundamental coweight, which is minuscule if and only if the corresponding node is conjugate to the affine node in the associated affine Dynkin diagram by a diagram automorphism. 

Since $\bfN$ is a $G$-representation, 
$$w\Eu(z^\lambda \bfN_\calO/(z^\lambda \bfN_\calO\cap \bfN_\calO))
=\Eu(z^{w\lambda}\bfN_\calO/(z^{w\lambda} \bfN_\calO\cap \bfN_\calO)).$$
For a general cocharacter $\lambda$ and any $T$-invariant one-dimensional summand $S$ of $\bfN$ of weight $\alpha$, 
$$z^\lambda S_\calO = z^{\lambda}S\otimes \calO=S\otimes z^m\calO,$$
where $m=\langle \lambda, \alpha\rangle$.
Hence,
$$z^{\lambda}S_\calO/(z^{\lambda}S_\calO\cap S_\calO)
=S\otimes \begin{cases}
\{0\}, & \textit{ if } m\geq 0,\\
\mathbb{C}z^{-1}\oplus \cdots \oplus \mathbb{C}z^{m}, & \textit{ if } m< 0,
\end{cases}$$
and 
$$\Eu(z^{\lambda}S_\calO/(z^{\lambda}S_\calO\cap S_\calO))
=\begin{cases}
1 , & m\geq 0,\\
(\alpha-\tfrac{\hbar}{2})
(\alpha-\tfrac{3\hbar}{2})
\cdots
(\alpha-\tfrac{(2|m|-1)\hbar}{2})
, & m< 0.
\end{cases}$$
Recall that $S\subset \bfN$ is assumed to have $\mathbb{C}_\hbar^\times$-weight $\hbar/2$.

Let $F$ be a reductive group and $\tilde{G}=G\times F$. 
We assume further that the $G$-representation $\bfN$ can be extended to a $\tilde{G}$-representation. 
Then we can slightly extend the above definition by 
$$\mathcal{A}_\hbar := 
H^{(G_{\calO}\times F_{\calO})\rtimes\mathbb{C}_\hbar^\times}_*(\calR_{G,\bfN}).$$
The dependence of $F$ will be clear from the context. 
Then the embedding \eqref{eq:embeddiff} becomes
$$\mathcal{A}_\hbar\hookrightarrow H_F^*(\pt)\otimes\Diff_\hbar(T).$$
The results in this section still hold.

\subsection{Quiver with involution}\label{sec:quiverinv}
Let $Q=(Q_0,Q_1,s,t)$ be a quiver, where $Q_0$ is the set of vertices while $Q_1$ is the set of arrows. For any $h\in Q_1$, $s(h)$ (resp. $t(h)$) denotes the source (resp. target) of the edge $h$.
For two $Q_0$-graded vector spaces $W,V$ with dimension vectors $\mathbf{w}=(w_i)_{i\in Q_0},\mathbf{v}=(v_i)_{i\in Q_0}$, 
let us denote 
\begin{align*}
G_V & =\prod_{i\in Q_0}GL(V_i),&
G_W & = \prod_{i\in Q_0}GL(W_i),\\
E_V & =\bigoplus_{h\in Q_1}\Hom(V_{s(h)},V_{t(h)}),&
L_{W,V} & = \bigoplus_{i\in Q_0}\Hom(W_i,V_i).
\end{align*}
The data
$$G=G_V,\qquad \bfN = E_V\oplus L_{W,V},\qquad F=G_W$$
gives rise to a \emph{quiver gauge theory}, whose Higgs branch is the Nakajima quiver variety and
its Coulomb branch algebra was computed to be the truncated shifted Yangians, see \cite[Appendix]{BFN19} and \cite{We19}.

Now assume that there is an involution $\tau$ on $Q=(Q_0,Q_1,s,t)$ satisfying
\begin{itemize}
    \item 
    $s(\tau(h))=\tau(t(h))$ and $ t(\tau(h))=\tau(s(h))$; 
    \item $\tau(s(h))=t(h)$ if and only if $\tau(h)=h$. 
\end{itemize}
To avoid confusion, we always draw the involution $\tau$ by red arrows.

We remark that the pair $(Q,\tau)$ is called a \emph{quiver with involution} in the literature; see \cite{EK07} and \cite{VV11}. 
Assume $\bigoplus_{i\in Q_0}V_i$ is equipped with a nondegenerate symmetric bilinear form $\langle-,-\rangle$ such that the orthogonal complement
$V_i^\perp=\bigoplus_{j\neq \tau i} V_j$. 
In particular, $\langle-,-\rangle$ restricts to a perfect pairing between $V_i$ and $V_{\tau i}$. Hence we can identify $V_i=V_{\tau i}^*$ and $\mathbf{v}$ is $\tau$-invariant. 
Let us denote 
\begin{align*}
G_V^{\tau} &=\left\{g\in G_V\mid g_i^t = g_{\tau i}^{-1}, \forall i\in Q_0\right\}\subseteq G_V,\\
E_V^\tau   &=\left\{f\in E_V\mid f_{\tau h}=-f_h^t, \forall h\in Q_1 \right\}\subseteq E_V.
\end{align*}
Hence, $G_V^\tau$ acts on $E_V^\tau$.
From now on, we will make two technical assumptions. We first assume 
\begin{equation}\label{equ:simplylaced}
\text{the quiver $Q$ is simply-laced without self-loop.}
\end{equation}
Following \cite{En09} and \cite[Section 4.1]{VV11}, we also put the following constraint on the involution
\begin{equation}\label{eq:nofixedpoint}
\tau i\neq i\text{ for any $i\in Q_0$}.
\end{equation}
With these assumptions, we will study the Coulomb branch algebra for
$$G=G_V^{\tau},\qquad \bfN = E_V^\tau \oplus L_{W,V},\qquad 
F=G_W. $$

In order to make the computation more precise, we can pick a decomposition $Q_0=Q_0^{+}\sqcup Q_0^-$ such that $i\in Q_0^{+}$ if and only if $\tau i\in Q_0^-$. 
Then the composition 
$$G_V^{\tau}\hookrightarrow G_V
\stackrel{\text{pr}}\longrightarrow \prod_{i\in Q_0^+}GL(V_i)$$
is an isomorphism. 
Let $Q_1^\tau$ be the subset of edges in $Q_1$, which are fixed by the involution $\tau$. 
For a decomposition 
$Q_1\setminus Q_1^\tau=Q_1^+ \sqcup Q_1^-$ such that $h\in Q_1^{+}$ if and only if $\tau(h)\in Q_1^-$, 
the composition
$$E_V^\tau \hookrightarrow E_V
\stackrel{\text{pr}}\longrightarrow 
\bigoplus_{h\in Q_1^+}
\Hom(V_{s(h)},V_{t(h)})
\oplus 
\bigoplus_{h\in Q_1^\tau}
\operatorname{Alt}(V_{s(h)},V_{t(h)})
$$
is an isomorphism, where 
$$\operatorname{Alt}(V_{s(h)},V_{t(h)})
=\mathsf{\Lambda}^2 V_{s(h)}^*
=\mathsf{\Lambda}^2 V_{t(h)}
$$
can be viewed as a quotient of $\Hom(V_{s(h)},V_{t(h)})$ if $h\in Q_1^\tau$. 

\begin{eg}[Diagonal Type]
Assume that the quiver $Q$ is the disjoint union of two copies of quivers with identical vertices and opposite orientation of arrows, denoted $Q^+$ and $Q^-$. 
Let $\tau: Q \to Q$ be the involution that interchanges the vertices in corresponding positions. 
$$
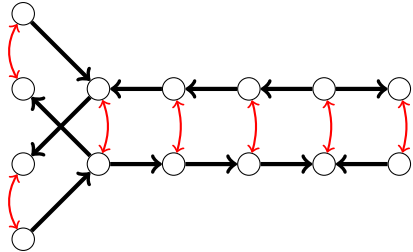
\begin{tikzpicture}[
  scale=1,
  every node/.style={circle, draw, inner sep=3pt},
  arrow/.style={red, <->, thick}
]

\node (t1) at (-1,1) {};   
\node (t2) at (0,0) {};   
\node (t3) at (-1,-1) {};  
\node (t4) at (1,0) {};
\node (t5) at (2,0) {};
\node (t6) at (3,0) {};
\node (t7) at (4,0) {};

\draw[->,ultra thick] (t1) -- (t2); 
\draw[->,ultra thick] (t2) -- (t3);
\draw[<-,ultra thick] (t2) -- (t4);
\draw[<-,ultra thick] (t4) -- (t5);
\draw[<-,ultra thick] (t5) -- (t6);
\draw[->,ultra thick] (t6) -- (t7);

\node (b1) at (-1,0) {};   
\node (b2) at (0,-1) {};   
\node (b3) at (-1,-2) {};   
\node (b4) at (1,-1) {};
\node (b5) at (2,-1) {};
\node (b6) at (3,-1) {};
\node (b7) at (4,-1) {};

\draw[<-,ultra thick] (b1) -- (b2);
\draw[<-,ultra thick] (b2) -- (b3);
\draw[->,ultra thick] (b2) -- (b4);
\draw[->,ultra thick] (b4) -- (b5);
\draw[->,ultra thick] (b5) -- (b6);
\draw[<-,ultra thick] (b6) -- (b7);

\draw[arrow] (t1) to[bend right=30] (b1);
\draw[arrow] (t2) to[bend left=20] (b2);
\draw[arrow] (t3) to[bend right=30] (b3);
\draw[arrow] (t4) to[bend left=15] (b4);
\draw[arrow] (t5) to[bend left=15] (b5);
\draw[arrow] (t6) to[bend left=15] (b6);
\draw[arrow] (t7) to[bend left=15] (b7);

\end{tikzpicture}$$
Let $V^+ = \bigoplus_{i\in Q_0^+}V_i$ be the $Q_0^+$-graded vector space.
Then we have isomorphisms
$$G_V^{\tau} \simeq G_{V^+},\qquad E_V^\tau \simeq E_{V^+}.$$
\end{eg}

\begin{eg}[Type AIII]
Consider a quiver with involution obtained from a Satake diagram of type AIII (see \cite{Ara62}), i.e. 
$$Q_0=\{1,\ldots,2n\},\qquad 
\tau i = 2n+1-i$$
and $Q_1$ is chosen such that 
$$
\texttt{\#}\big\{h\in Q_1:\{s(h),t(h)\}=\{i,j\}\big\}=\delta_{|i-j|,1}.$$
We can choose $Q_0^+=\{1,\ldots,n\}$ and $Q_1^+=\{h\in Q_1: \max\{s(h),t(h)\}\leq n\}$.
For example, when $n=3$ we have
$$
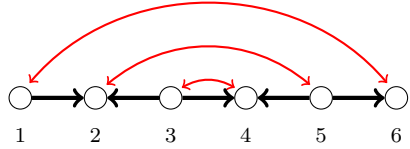
\begin{tikzpicture}[scale=1, 
                    arrow/.style={red, <->, thick}]
\foreach \i in {0,...,5}
    \node[circle, draw, inner sep=3pt] (a\i) at (\i,0) {};
    \foreach \i in {0,...,5}
    \node (b\i) at (\i,.05) {};
    \foreach \i in {1,...,6}
    \node  at (\i-1,-.5) {$\scriptstyle \i$};
\draw[arrow] (b0) to[bend left=50] (b5);
\draw[arrow] (b1) to[bend left=40] (b4);
\draw[arrow] (b2) to[bend left=30] (b3);
\draw[->,ultra thick] (a0) -- (a1); 
\draw[->,ultra thick] (a2) -- (a1); 
\draw[->,ultra thick] (a2) -- (a3); 
\draw[->,ultra thick] (a4) -- (a3); 
\draw[->,ultra thick] (a4) -- (a5); 
\end{tikzpicture}
$$
Then we have 
\begin{align*}
G_V^{\tau} & \cong GL(v_1)\times GL(v_2)\times GL(v_3)\\
E_V^\tau & \cong \Hom(\mathbb{C}^{v_1},\mathbb{C}^{v_2})\oplus 
\Hom(\mathbb{C}^{v_3},\mathbb{C}^{v_2})\oplus 
\wedge^2(\mathbb{C}^{v_3}).
\end{align*}
\end{eg}

\begin{eg}
We remark that the quiver can be very general, for example, it can be of the following types:
$$
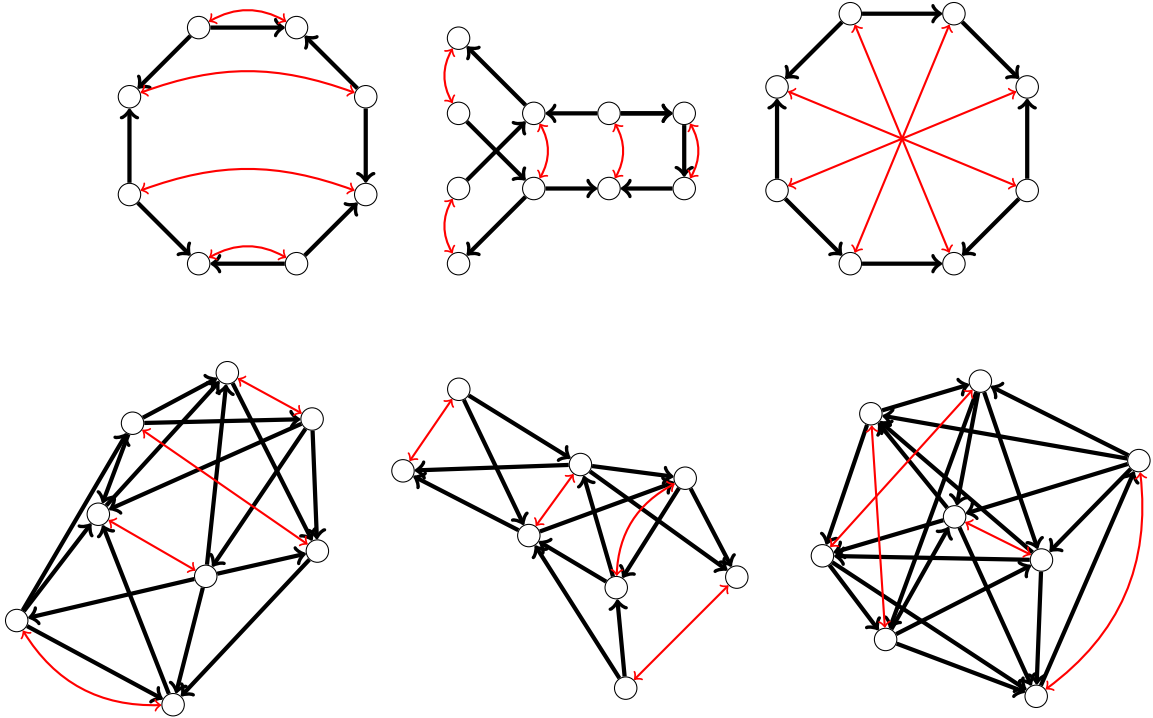
\begin{tikzpicture}[
    every node/.style={circle, draw, inner sep=3pt},
    arrow/.style={red, <->, thick},scale=1
]

\def\n{8}
\def\rx{1.7}  
\def\ry{1.7}  

\foreach \i in {0,...,7} {
  \node (a\i) at ({\rx*cos(360/\n * \i+180/\n)}, {\ry*sin(360/\n * \i + 180/\n)}) {};
}

  \draw[->,ultra thick] (a0) -- (a1);
  \draw[<-,ultra thick] (a1) -- (a2);
  \draw[->,ultra thick] (a2) -- (a3);
  \draw[<-,ultra thick] (a3) -- (a4);
  \draw[->,ultra thick] (a4) -- (a5);
  \draw[<-,ultra thick] (a5) -- (a6);
  \draw[->,ultra thick] (a6) -- (a7);
  \draw[->,ultra thick] (a0) -- (a7);

\draw[arrow] (a1) to [bend right = 30] (a2);
\draw[arrow] (a0) to [bend right = 20] (a3);
\draw[arrow] (a4) to [bend right = -20] (a7);
\draw[arrow] (a5) to [bend right = -30] (a6);
\end{tikzpicture}
\qquad
\begin{tikzpicture}[scale=1, every node/.style={circle, draw, inner sep=3pt}, 
                    arrow/.style={red, <->, thick}]
\node(a0) at (0,0) {};
\node(a1) at (1,0) {};
\node(a2) at (2,0) {};
\node(a3) at (2,-1) {};
\node(a4) at (1,-1) {};
\node(a5) at (0,-1) {};
\node(NE) at (-1,0) {};
\node(SE) at (-1,-2) {};
\node(NW) at (-1,1) {};
\node(SW) at (-1,-1) {};
\draw[->,ultra thick] (a0) -- (NW); 
\draw[<-,ultra thick] (a0) -- (SW); 
\draw[<-,ultra thick] (a0) -- (a1); 
\draw[->,ultra thick] (a1) -- (a2); 
\draw[->,ultra thick] (a1) -- (a2); 
\draw[->,ultra thick] (a2) -- (a3); 
\draw[->,ultra thick] (a3) -- (a4); 
\draw[<-,ultra thick] (a4) -- (a5); 
\draw[<-,ultra thick] (a5) -- (NE); 
\draw[->,ultra thick] (a5) -- (SE); 
\draw[arrow] (a0) to[bend left=30] (a5);
\draw[arrow] (a1) to[bend left=30] (a4);
\draw[arrow] (a2) to[bend left=30] (a3);
\draw[arrow] (NW) to[bend left=-30] (NE);
\draw[arrow] (SW) to[bend left=-30] (SE);
\end{tikzpicture}
\qquad
\begin{tikzpicture}[scale=1,
    every node/.style={circle, draw, inner sep=3pt},
    arrow/.style={red, <->, thick}
]

\def\n{8}
\def\rx{1.8}  
\def\ry{1.8}  

\foreach \i in {0,...,7} {
  \node (a\i) at ({\rx*cos(360/\n * \i +180/\n)}, {\ry*sin(360/\n * \i+180/\n)}) {};
}

  \draw[<-,ultra thick] (a0) -- (a1);
  \draw[<-,ultra thick] (a1) -- (a2);
  \draw[->,ultra thick] (a2) -- (a3);
  \draw[<-,ultra thick] (a3) -- (a4);
  \draw[->,ultra thick] (a4) -- (a5);
  \draw[->,ultra thick] (a5) -- (a6);
  \draw[<-,ultra thick] (a6) -- (a7);
  \draw[->,ultra thick] (a7) -- (a0);

\foreach \i in {0,...,3} {
  \pgfmathtruncatemacro{\j}{mod(\i+4,\n)}
  \draw[arrow] (a\i) to (a\j);
}
\end{tikzpicture}
$$
\end{eg}

$$\begin{matrix}\begin{tikzpicture}[scale=.7, every node/.style={circle, draw, inner sep=3pt}] 
\node(a0) at (5.90,1.89) {}; \node(a1) at (7.70,2.85) {}; \node(a2) at (7.29,-1.02) {}; \node(a3) at (6.67,-3.46) {}; \node(a4) at (3.70,-1.86) {}; \node(a5) at (5.25,0.16) {}; \node(a6) at (9.31,1.97) {}; \node(a7) at (9.41,-0.54) {}; 
\draw[->,ultra thick] (a4) -- (a0); \draw[->,ultra thick] (a0) -- (a6); \draw[->,ultra thick] (a2) -- (a1); \draw[->,ultra thick] (a2) -- (a3); \draw[->,ultra thick] (a6) -- (a2); \draw[->,ultra thick] (a2) -- (a7); \draw[->,ultra thick] (a4) -- (a3); \draw[->,ultra thick] (a3) -- (a5); \draw[->,ultra thick] (a7) -- (a3); \draw[->,ultra thick] (a6) -- (a7); \draw[->,ultra thick] (a1) -- (a7); \draw[->,ultra thick] (a6) -- (a5); \draw[->,ultra thick] (a4) -- (a5); \draw[->,ultra thick] (a5) -- (a1); \draw[->,ultra thick] (a0) -- (a5); \draw[->,ultra thick] (a2) -- (a4); \draw[->,ultra thick] (a0) -- (a1); 
\draw[red, <->, thick] (a0) to (a7); 
\draw[red, <->, thick] (a1) to (a6); 
\draw[red, <->, thick] (a2) to (a5); 
\draw[red, <->, thick] (a3) to [bend left = 30] (a4);
\end{tikzpicture}\end{matrix}\qquad 
\begin{matrix}\begin{tikzpicture}[scale=.7, every node/.style={circle, draw, inner sep=3pt}] 
\node(a0) at (3.49,1.01) {}; \node(a1) at (7.72,-3.12) {}; \node(a2) at (7.54,-1.21) {}; \node(a3) at (5.88,-0.22) {}; \node(a4) at (6.86,1.13) {}; \node(a5) at (8.85,0.87) {}; \node(a6) at (9.83,-1.01) {}; \node(a7) at (4.55,2.56) {}; 
\draw[->,ultra thick] (a3) -- (a0); \draw[->,ultra thick] (a4) -- (a0); \draw[->,ultra thick] (a1) -- (a3); \draw[->,ultra thick] (a5) -- (a2); \draw[->,ultra thick] (a3) -- (a5); \draw[->,ultra thick] (a4) -- (a5); \draw[->,ultra thick] (a7) -- (a4); \draw[->,ultra thick] (a5) -- (a6); \draw[->,ultra thick] (a7) -- (a3); \draw[->,ultra thick] (a4) -- (a6); \draw[->,ultra thick] (a2) -- (a4); \draw[->,ultra thick] (a2) -- (a3); \draw[->,ultra thick] (a1) -- (a2); 
\draw[red, <->, thick] (a0) to (a7); 
\draw[red, <->, thick] (a1) to (a6); 
\draw[red, <->, thick] (a2) to [bend left = 30] (a5);
\draw[red, <->, thick] (a3) to (a4); 
\end{tikzpicture}\end{matrix}\qquad 
\begin{matrix}\begin{tikzpicture}[scale=.7, every node/.style={circle, draw, inner sep=3pt}] 
\node(a0) at (7.68,-0.61) {}; \node(a1) at (7.58,-3.20) {}; \node(a2) at (6.52,2.79) {}; \node(a3) at (4.44,2.17) {}; \node(a4) at (4.72,-2.12) {}; \node(a5) at (3.52,-0.53) {}; \node(a6) at (9.53,1.28) {}; \node(a7) at (6.03,0.21) {}; 
\draw[->,ultra thick] (a0) -- (a1); \draw[->,ultra thick] (a2) -- (a0); \draw[->,ultra thick] (a0) -- (a3); \draw[->,ultra thick] (a4) -- (a0); \draw[->,ultra thick] (a4) -- (a1); \draw[->,ultra thick] (a1) -- (a6); \draw[->,ultra thick] (a7) -- (a1); \draw[->,ultra thick] (a2) -- (a4); \draw[->,ultra thick] (a6) -- (a2); \draw[->,ultra thick] (a2) -- (a7); \draw[->,ultra thick] (a3) -- (a5); \draw[->,ultra thick] (a7) -- (a3); \draw[->,ultra thick] (a5) -- (a4); \draw[->,ultra thick] (a6) -- (a7); \draw[->,ultra thick] (a7) -- (a5); \draw[->,ultra thick] (a4) -- (a7); \draw[->,ultra thick] (a6) -- (a3); \draw[->,ultra thick] (a6) -- (a0); \draw[->,ultra thick] (a5) -- (a1); \draw[->,ultra thick] (a0) -- (a5); \draw[->,ultra thick] (a3) -- (a2); 
\draw[red, <->, thick] (a0) to (a7); 
\draw[red, <->, thick] (a1) to [bend right = 30] (a6);
\draw[red, <->, thick] (a2) to (a5); 
\draw[red, <->, thick] (a3) to (a4); 
\end{tikzpicture}\end{matrix}
$$

Let us pick a basis $\{e_{i,1},\ldots,e_{i,v_i}\}$ of each $V_i$ such that $\langle e_{i_1,j_1},e_{i_2,j_2}\rangle = \delta_{i_1,\tau i_2 }\delta_{j_1,j_2}$. 
This gives a choice of a maximal torus $T_V^\tau$ of $G_V^{\tau}$. Then we can identify 
\begin{align*}
H_{T_V^\tau}^*(\pt)
& =\bigotimes_{i\in Q_0^+}
\mathbb{Q}[x_{i,1},\ldots,x_{i,v_i}],
\end{align*}
and
$$X_*(T^\tau_V) = 
\bigoplus_{i\in Q_0^+}\mathbb{Z}\epsilon_{i,1}
\oplus \cdots \oplus \mathbb{Z}\epsilon_{i,v_i}.$$
Here $\epsilon_{i,j}$ is the $j$-th coordinate on the subtorus $T_{V_i}$ and $x_{i,j}$ is the $j$-th character of $\Lie(T_{V_i})$. A cocharacter $\lambda=\sum\lambda_{i,j}\epsilon_{i,j}\in X_*(T^\tau_V)$ is dominant if 
$\lambda_{i,1}\geq \cdots \geq \lambda_{i,v_i}$ for each $i\in Q_0^+$. 
For $i\in Q_0^+$ and $1\leq j\leq v_i$, we denote the difference operator of $\epsilon_{i,j}$ by $d_{i,j}\in \Diff_\hbar(T_V^\tau)$. 
Let us introduce $x_{i,j}=-x_{\tau i,j}$ and $d_{i,j}=d_{\tau i,j}^{-1}$ for $i\in Q_0^-$. 
Hence, we have the following relations
$$
d_{i_1,j_1}x_{i_2,j_2}
=
(x_{i_2,j_2}+
(\delta_{i_1,i_2}
-\delta_{i_1,\tau i_2 })
\delta_{j_1,j_2}\hbar
)d_{i_1,j_1}
$$
for any $i_1,i_2\in Q_0$, $1\leq j_1\leq v_{i_1}$ and $1\leq j_2\leq v_{i_2}$. 
Let us pick a basis $\{f_{i,1},\ldots,f_{i,w_i}\}$ of each $W_i$. This gives a choice of a maximal torus $T_W$ of $G_W$, and we can identify 
$$H_{G_W}^*(\pt)
=\bigotimes_{i\in Q_0}
\mathbb{Q}[w_{i,1},\ldots,w_{i,w_i}]^{S_{w_i}}. 
$$

\subsection{Monopole operators}\label{sec:mono}
In this section, we give explicit formulae for some monopole operators of the Coulomb branch algebra $$\mathcal{A}_\hbar := 
H^{(G_{\calO}\times F_{\calO})\rtimes\mathbb{C}_\hbar^\times}_*(\calR_{G,\bfN})$$ associated to 
$$G=G_V^{\tau} \simeq \prod_{i\in Q_0^+}GL(V_i),\qquad \bfN = E_V^\tau \oplus L_{W,V},\qquad 
F=G_W. $$
Recall that it can be embedded into the difference algebra $H_{G_W}^*(\pt)\otimes\Diff_\hbar(T_V^\tau)$.

For $i\in Q_0$,
define the following polynomials
$$V_i(z) = \prod_{k=1}^{v_i}(z-x_{i,k}), \textit{ and \quad} 
W_i(z) = \prod_{k=1}^{w_i}(z-w_{i,k}).$$
Moreover, for $1\leq r\neq s\leq v_i$, let $$V_{i,r}(z)=\frac{V_i(z)}{z-x_{i,r}}, \textit{ and \quad} V_{i,\{r,s\}}(z)=\frac{V_i(z)}{(z-x_{i,r})(z-x_{i,s})}.$$

\def\r{{\color{red}r}}

For $i\in Q_0^+$, the dominant coweight $\epsilon_{i,1}\in X_*(T_V^\tau)$ is a minuscule coweight for $G$. Hence, the spherical Schubert cell $\Gr_{\epsilon_{i,1}}$ is closed and isomorphic to the projective space $\mathbb{P}^{v_i-1}$. More precisely, it is identified with the moduli space of $\calO$-modules $L$ such that 
\[z\calO\otimes V_i\subset L\subset \calO\otimes V_i,\quad \dim_\mathbb{C}\calO\otimes V_i/L =1.\]
There is a tautological line bundle on $\Gr_{\epsilon_{i,1}}$ whose fiber at $L$ is $\calO\otimes V_i/L$. Thus, the torus weight of this line bundle at the fixed point $z^{\epsilon_{i,r}}$ is $x_{i,r}$. Let $\mathcal{Q}_i$ denote the pullback of this line bundle to $\calR_{\epsilon_{i,1}}$.  
\begin{prop}\label{prop:fmono}
Let $f\in \mathbb{Q}[x]$ be a polynomial in one variable. We have 
\begin{align*}
    &f(c_1(\mathcal{Q}_i)) \cap [\mathcal{R}_{\epsilon_{i,1}}]
    =\sum_{r=1}^{v_i}f(x_{i,r})
    \prod_{h\in Q_1\setminus Q_1^\tau \atop s(h)=i} (-1)^{v_{t(h)}}V_{t(h)}(x_{i,r}+\tfrac{\hbar}{2})\\
    &\cdot\prod_{h\in Q_1^\tau\atop s(h)=i}(-1)^{v_i-1}V_{\tau i,r}(x_{i,r}+\tfrac{\hbar}{2})\cdot \frac{W_{\tau i}(x_{\tau i,r}-\tfrac{\hbar}{2})}{V_{i,r}(x_{i,r})}d_{i,r}\in H_{G_W}^*(\pt)\otimes\Diff_\hbar(T_V^\tau).
\end{align*}
\end{prop}
\begin{proof}
For $1\leq r\leq v_i$, \eqref{equ:wts} shows
$$\Eu(T_{\epsilon_{i,r}} \Gr_{\epsilon_{i,1}})
=\prod_{s\neq r}(x_{i,r}-x_{i,s}) = 
V_{i,r}(x_{i,r}).$$
Let us compute the space 
$z^{\epsilon_{i,r}}\bfN/(z^{\epsilon_{i,r}}\bfN\cap \bfN)$. Recall that $\mathbb{C}^\times_\hbar$ acts on $N$ by weight $\frac{\hbar}{2}$.
We can decompose $\bfN$ into one-dimensional $T_V^\tau$-invariant subspaces
\begin{align}
\bfN & = \bigoplus_{h\in Q_1^+} 
\bigoplus_{1\leq a\leq v_{s(h)}}
\bigoplus_{1\leq b\leq v_{t(h)}}
\Hom(\mathbb{C}e_{s(h),a},\mathbb{C}e_{t(h),b})
\label{eq:HomofQ1+}\\
& \oplus 
\bigoplus_{h\in Q_1^\tau}
\bigoplus_{1\leq a<b\leq v_{t(h)}}
\mathbb{C} (e_{t(h),a}\wedge e_{t(h),b})
\label{eq:HomofQ10}\\
& \oplus 
\bigoplus_{j\in Q_0}
\bigoplus_{1\leq a\leq w_j}
\bigoplus_{1\leq b\leq v_j}
\Hom(\mathbb{C}f_{j,a},\mathbb{C}e_{j,b}).
\label{eq:Homofframe}
\end{align}
We can compute the contribution of 
$$z^{\epsilon_{i,r}}\bfN_\calO/(z^{\epsilon_{i,r}}\bfN_\calO\cap \bfN_\calO)$$
for each factor:
\begin{itemize}
    \item 
For $S = \Hom(\mathbb{C}e_{s(h),a},\mathbb{C}e_{t(h),b})$ in \eqref{eq:HomofQ1+}, we have 
$z^{\epsilon_{i,r}}S_\calO=S\otimes z^m\calO$ with 
$$m=
-(\delta_{s(h),i}-\delta_{s(h),\tau i})\delta_{a,r}
+(\delta_{t(h),i}-\delta_{t(h),\tau i})\delta_{b,r}.
$$
We have $m<0$ only when 
$$s(h)=i,a=r\qquad \text{or}\qquad 
t(h)=\tau i,b=r,$$
and in both cases, $m=-1$. 

The total contribution of them is 
\begin{align*}
& \quad \prod_{h\in Q_1^+\atop s(h)=i}
\prod_{b=1}^{v_{t(h)}}
(x_{t(h),b}-x_{i,r}-\tfrac{\hbar}{2})
\prod_{h\in Q_1^+\atop t(h)=\tau i}
\prod_{a=1}^{v_{s(h)}}
(x_{\tau i,r}-x_{s(h),a}-\tfrac{\hbar}{2}). 
\end{align*}
Notice that
\begin{align*}
    \prod_{h\in Q_1^+\atop t(h)=\tau i}
\prod_{a=1}^{v_{s(h)}}
(x_{\tau i,r}-x_{s(h),a}-\tfrac{\hbar}{2})=&\prod_{h\in Q_1^-\atop s(h)=i}
\prod_{a=1}^{v_{t(h)}}
(x_{\tau i,r}-x_{\tau(t(h)),a}-\tfrac{\hbar}{2})\\
=&\prod_{h\in Q_1^-\atop s(h)=i}
\prod_{a=1}^{v_{t(h)}}
(x_{t(h),a}-x_{i,r}-\tfrac{\hbar}{2}).
\end{align*}
Here we used the fact $x_{\tau i,r}=-x_{i,r}$.
Thus, the total contribution is 
\[\prod_{h\in Q_1\setminus Q_1^\tau\atop s(h)=i}\prod_{b=1}^{v_{t(h)}}
(x_{t(h),b}-x_{i,r}-\tfrac{\hbar}{2})
=\prod_{h\in Q_1\setminus Q_1^\tau\atop s(h)=i} (-1)^{v_{t(h)}}
V_{t(h)}(x_{i,r}+\tfrac{\hbar}{2}).\]

    \item 
For $S = \mathbb{C}e_{t(h),a}\wedge e_{t(h),b}$ in \eqref{eq:HomofQ10}, we have 
$z^{\epsilon_{i,r}}S_\calO=S\otimes z^m\calO$ with 
$$m=
(-\delta_{s(h),i}+\delta_{t(h),i})(\delta_{a,r}+\delta_{b,r}).$$
Here we have used the fact that $s(h)=\tau(t(h))$ as $h\in Q_1^\tau$.
Hence, $m<0$ only when 
$$s(h)=i,r\in \{a,b\},$$
and in both cases, $m=-1$. Similar as above, the total contribution is \begin{align*}
&\quad \prod_{h\in Q_1^\tau\atop s(h)=i}\prod_{1\leq a\leq v_{t(h)}\atop a\neq r}(-x_{i,a}-x_{i,r}-\tfrac{\hbar}{2})
 =\prod_{h\in Q_1^\tau\atop s(h)=i}\prod_{1\leq a\leq v_{t(h)}\atop a\neq r}(x_{t(h),a}-x_{i,r}-\tfrac{\hbar}{2})\\
& = \prod_{h\in Q_1^\tau\atop s(h)=i}
(-1)^{v_i-1}
V_{t(h),r}(x_{i,r}+\tfrac{\hbar}{2}). 
\end{align*}
    \item 
For $S=\Hom(\mathbb{C}f_{j,a},\mathbb{C}e_{j,b})$ in 
\eqref{eq:Homofframe}, we have 
$z^{\epsilon_{i,r}}S_\calO=S\otimes z^m\calO$ with 
$$m=
(\delta_{j,i}-\delta_{j,\tau i})\delta_{b,r}.$$
Hence, $m<0$ only if $j=\tau i$ and $b=r$, and in this case $m=-1$. Therefore, the total contribution is 
$$\prod_{1\leq a\leq w_{\tau i}}(-w_{\tau i,a}+x_{\tau i,r}-\tfrac{\hbar}{2})
= W_{\tau i}(x_{\tau i,r}-\tfrac{\hbar}{2}).$$
\end{itemize}
Substituting all of these into \eqref{eq:minu-mono-op}, we get the desired formula. 
\end{proof}

Now let us consider the case of $\lambda=-\epsilon_{i,v_i}\in X_*(T_V^\tau)$, where $i\in Q_0^+$. This is also a minuscule coweight, and the spherical Schubert cell $\Gr_{-\epsilon_{i,v_i}}$ is also isomorphic to the projective space $\mathbb{P}^{v_i-1}$. More precisely, it is identified with the moduli space of $\calO$-modules $L$ such that 
\[\calO\otimes V_i\subset L\subset z^{-1}\calO\otimes V_i,\quad \dim_\mathbb{C}L/(\calO\otimes V_i) =1.\]
There is a tautological line bundle on $\Gr_{-\epsilon_{i,v_i}}$ whose fiber at $L$ is $L/(\calO\otimes V_i)$. Therefore, the torus weight of the fiber at the torus fixed point $z^{-\epsilon_{i,r}}$ is $x_{i,r}-\hbar$. Let $\mathcal{S}_i$ denote the pullback of this line bundle to $\calR_{-\epsilon_{i,v_i}}$. Similar to the previous proposition, we have the following formula, whose proof is almost the same. 
\begin{prop}\label{prop:fmono2}
Let $f\in \mathbb{Q}[x]$ be a polynomial in one variable. We have 
\begin{align*}
    &f(c_1(\mathcal{S}_i)) \cap [\mathcal{R}_{-\epsilon_{i,v_i}}]
    =\sum_{r=1}^{v_i}f(x_{i,r}-\hbar)
    \prod_{h\in Q_1\setminus Q_1^\tau\atop s(h)=\tau i}(-1)^{v_{t(h)}}V_{t(h)}(x_{\tau i,r}+\tfrac{\hbar}{2})\\
    &\cdot\prod_{h\in Q_1^\tau\atop t(h)=i}(-1)^{v_i-1}V_{i,r}(x_{\tau i,r}+\tfrac{\hbar}{2})\cdot \frac{W_{i}(x_{i,r}-\tfrac{\hbar}{2})}{(-1)^{v_i-1}V_{i,r}(x_{i,r})}d_{i,r}^{-1}\in H_{G_W}^*(\pt)\otimes\Diff_\hbar(T_V^\tau).
\end{align*}
\end{prop}

\section{Shifted twisted Yangians}\label{sec:Yangian}

Let $C = (c_{ij})_{i,j \in I}$ be a symmetric generalized Cartan matrix, and let $\mathfrak{g}$ be the associated Kac--Moody Lie algebra with Cartan subalgebra $\mathfrak{h}$. 
We further make the following assumptions: 
\begin{itemize}
    \item there exists an involution $\tau: I \to I$ such that $c_{ij} = c_{\tau i, \tau j}$ for all $i,j \in I$;
    \item 
    $c_{ij} = c_{ji} \in \{0,-1\}$, for any $i \neq j$;
    \item $i \neq \tau i$ for all $i \in I$.
\end{itemize}

Let $\{\alpha_i\in \mathfrak{h}^* \mid i \in I\}$ (resp. $\{\alpha_i^\vee\in \mathfrak{h} \mid i \in I\}$) be the set of simple roots (coroots)
associated with $\mathfrak{g}$. Let $\Lambda_i\in \mathfrak{h}$ be the fundamental coweights satisfying \[
  \langle  \Lambda_i,\alpha_j  \rangle = \delta_{ij}, \qquad i,j \in I,
\]
where $\langle-,-\rangle$ is the natural pairing between $\mathfrak{h}$ and $\mathfrak{h}^*$.
Let $P^\vee\subset \mathfrak{h}$ be the lattice generated by $\alpha_i^\vee$ and $\Lambda_i$.
The involution $\tau$ on $I$ naturally extends to $P^\vee$ by setting $\tau(\Lambda_i) = \Lambda_{\tau i}$ and $\tau(\alpha_i^\vee)=\alpha_{\tau i}^\vee$.

The definition of the shifted twisted Yangian depends on a choice of a $\tau$-invariant coweight 
$\mu \in P^\vee$, i.e.\ $\mu = \tau(\mu)$. 
In order to write down explicit generators, we further fix a family of parameters 
$\zeta = (\zeta_i)_{i \in I} \in \bigl(\mathbb{C}(\hbar)^\times\bigr)^I$ satisfying 
\begin{equation}\label{eq:zeta}
  \zeta_i = (-1)^{\langle \alpha_i, \mu \rangle}\,\zeta_{\tau(i)} 
  \qquad \in \mathbb{C}(\hbar)^\times.
\end{equation}
The condition \eqref{eq:zeta} is imposed to ensure that the convention 
\eqref{eq:vanishing} is compatible with the second equation in
\eqref{eq:hcomm}. 
In fact, the resulting algebra structure is independent of the particular choice of~$\zeta$, see Remark \ref{rem:parameter} below. Throughout the paper, we define 
\[
\{x,y\}:=xy+yx.
\]
\begin{defin}
\label{shiftedTY}
The shifted twisted Yangian $\mathbf{Y}^{\tau}_{\mu}(\mathfrak{g})$ is the $\mathbb{C}(\hbar)$-algebra generated by 
$h_{i,r}$ and $b_{i,s}$ for $i\in I, r\geq -\langle \alpha_i,\mu\rangle-1$ and $s\geq 0$, subject to the following relations:
\begin{align}
\label{eq:hcomm}
    & {\left[h_{i, r}, h_{j, s}\right]=0, \qquad 
    h_{i,s} = (-1)^{s+1} h_{\tau i,s}}, 
        \\
\label{eq:hbcomm}
    & {\left[h_{i, r+2}, b_{j, s}\right]-\left[h_{i, r}, b_{j, s+2}\right]} 
    =\frac{c_{i j}-c_{\tau i, j}}{2} \hbar\left\{h_{i, r+1}, b_{j, s}\right\}
    \\\notag  & \qquad 
    +\frac{c_{i j}+c_{\tau i, j}}{2} \hbar\left\{h_{i, r}, b_{j, s+1}\right\}
    +\frac{c_{i j} c_{\tau i, j}}{4} \hbar^2\left[h_{i, r}, b_{j, s}\right], 
        \\
\label{eq:bcomm}
    & {\left[b_{i, r+1}, b_{j, s}\right]-\left[b_{i, r}, b_{j, s+1}\right]=\frac{c_{i j}}{2}\hbar \left\{b_{i, r}, b_{j, s}\right\}-2 \delta_{\tau i, j}(-1)^r h_{j, r+s+1}},
\end{align}
and the Serre-type relations  
\begin{itemize}
    \item when $c_{i j}=0$, we have 
\begin{equation}\label{eq:commSerre}
    \left[b_{i, r}, b_{j, s}\right]=\delta_{\tau i, j}(-1)^r h_{j, r+s};
\end{equation}
\item when $c_{i j}=-1, j \neq \tau i \neq i$, we have 
\begin{equation}\label{eq:usualSerre}
    \operatorname{Sym}_{k_1, k_2}\left[b_{i, k_1},\left[b_{i, k_2}, b_{j, r}\right]\right]=0;
\end{equation}
\item when $c_{i, \tau i}=-1$, we have 
\begin{equation}\label{eq:iSerre}
    \operatorname{Sym}_{k_1, k_2}\left[b_{i, k_1},\left[b_{i, k_2}, b_{\tau i, r}\right]\right]=\frac{4}{3} \operatorname{Sym}_{k_1, k_2}(-1)^{k_1} \sum_{p=0}^{\infty} 3^{-p}\left[b_{i, k_2+p}, h_{\tau i, k_1+r-p}\right].
\end{equation}
\end{itemize}
Here we have adopted the convention that 
\begin{equation}\label{eq:vanishing}
\mbox{$h_{i,r}=0$ if $r<-\langle \alpha_i,\mu\rangle-1$ and $h_{i,r}=\zeta_i$ if $r=-\langle \alpha_i,\mu\rangle-1$}
\end{equation}
\end{defin}
We note that the right-hand side of \cref{eq:iSerre} is a finite sum by the convention \cref{eq:vanishing}.

\begin{rem}
\label{rem:finiteSerre}
When \( \mu = 0 \) and \( \zeta_i = \hbar^{-1} \) for all \( i \in I \), the relations \cref{eq:hcomm}--\cref{eq:usualSerre} coincide with \cite[(3.1)--(3.7)]{LZ24}. Since we assume that the involution \( \tau \) on the set of vertices has no fixed points, the Serre-type relation \cite[(3.8)]{LZ24} does not appear in our setting. Moreover, our \cref{eq:iSerre} provides a generalization of \cite[(3.9)]{LZ24} to the context of shifted twisted Yangians. In particular, when \( \mu = 0 \), our \cref{eq:iSerre} recovers \cite[(3.9)]{LZ24}. This relation is obtained through direct computations on the Coulomb branch side; see \cref{iSerre(z)} for further details.
\end{rem}

\begin{rem}\label{rem:parameter} 
For any $\zeta\in \mathbb{C}(\hbar)^\times$ and $i\in I$, it is not hard to check that the following transformation 
\begin{align*}
h_{j,r} & \mapsto \zeta h_{j,r}\quad j\in \{i,\tau i\},
&\text{with other $h_{j,r}$ fixed}, 
\\
b_{i,s} & \mapsto \zeta b_{i,s}
&\text{with other $b_{j,s}$ fixed},
\end{align*}
preserves all the relations except \eqref{eq:vanishing}. 
In particular, up to an algebra automorphism, the shifted twisted Yangian does not depend on the choice of parameters. 
\end{rem}

\begin{rem}
The notion of shifted twisted Yangians was also introduced in \cite{TT24,LPTTW25}. Our construction and theirs arise from two distinct shifting procedures, yet the resulting algebras are expected to be isomorphic. 
For ordinary non-twisted Yangians, it is known that these two shifting methods give rise to isomorphic algebras (see \cite{KWWY14} and \cite[Remark B.3]{BFN19}), and the same phenomenon is conjectured to hold in the twisted setting. 
\end{rem}

We define the following generating functions
\begin{align}
\label{generatingfunction}
h_i(z) & = \hbar\sum_{r\in \mathbb{Z}}h_{i,r}z^{-r-1},&
h_i^\circ(z) & = \hbar\sum_{r\geq 0}h_{i,r}z^{-r-1},&
b_i(z) & = \hbar\sum_{s\geq 0}b_{i,s}z^{-s-1}.
\end{align}
Following \cite[\S 3.3]{LZ24}, we can rewrite the relations \eqref{eq:hcomm}, \eqref{eq:hbcomm} and \eqref{eq:bcomm} as 
\begin{align}
\label{eq:hcomm(z)}&
{\left[h_i(u), h_j(v)\right]=0, \qquad h_{\tau i}(u)=h_i(-u),}\\
\label{eq:hbcomm(z)}
&
\left(u^2-v^2\right)\left[h_i(u), b_j(v)\right]=\frac{c_{i j}-c_{\tau i, j}}{2} \hbar u\left\{h_i(u), b_j(v)\right\}\\
\notag& \qquad +\frac{c_{i j}+c_{\tau i, j}}{2} \hbar v\left\{h_i(u), b_j(v)\right\} +\frac{c_{i j} c_{\tau i, j}}{4} \hbar^2\left[h_i(u), b_j(v)\right]
\\\notag&\qquad 
-\hbar\left[h_i(u), b_{j, 1}\right] 
-\hbar v\left[h_i(u), b_{j, 0}\right]-\frac{c_{i j}+c_{\tau i, j}}{2} \hbar^2\left\{h_i(u), b_{j, 0}\right\},\\
\label{eq:bcomm(z)}&
(u-v)\left[b_i(u), b_j(v)\right]=\frac{c_{i j}}{2} \hbar\left\{b_i(u), b_j(v)\right\}+\hbar\left(\left[b_{i, 0}, b_j(v)\right]-\left[b_i(u), b_{j, 0}\right]\right) \\
\notag&\qquad -\delta_{\tau i, j} \hbar\left(\frac{2 u}{u+v} h_i^{\circ}(u)+\frac{2 v}{u+v} h_j^{\circ}(v)\right).
\end{align}
We can also write two of Serre relations \eqref{eq:commSerre} and \eqref{eq:usualSerre} as 
\begin{align}
\label{eq:commSerre(z)}&
(u+v)\left[b_i(u), b_j(v)\right]=\delta_{\tau i, j}\hbar\left(h_j^\circ(v)-h_i^\circ(u)\right)&(c_{ij}=0),\\
\label{eq:usualSerre(z)}&
\operatorname{Sym}_{u_1, u_2}\left[b_i(u_1),\left[b_i(u_2), b_j(v)\right]\right]=0, &
(c_{i j}=-1, j \neq \tau i \neq i).
\end{align}

\begin{rem}
\label{rem:iSerre}
The relations \cref{eq:hcomm(z)}--\cref{eq:usualSerre(z)} coincide with \cite[(3.20)--(3.24)]{LZ24}.
However, \cite{LZ24} did \textbf{not} 
write down a generating function formulation for the relation \cref{eq:iSerre}.
To address this issue, we establish an analogue of \cite[Section 3.4]{LZ24} for the relation \cref{eq:iSerre} on the Coulomb branch side; see \cref{sec:gftoiSerre} for further details.
\end{rem}

\section{Main Results}
In this section, we first produce a GKLO-type representation of the shifted twisted Yangian via difference operators. Then we prove our main result showing that for a quiver with involution, there is an algebra homomorphism from the shifted twisted Yangian to the quantized Coulomb branch algebra associated to the corresponding involution-fixed part of the quiver gauge theory.

Let $(Q,\tau)$ be a quiver with involution as in Section \ref{sec:quiverinv}, then there is a symmetric Cartan matrix with involution $(C,\tau)$ by forgetting the orientation of $Q$, i.e.
$$I=Q_0,\qquad \tau=\tau|_{Q_0},\qquad 
c_{ij}=-
\texttt{\#}\left\{h\in Q_1:\{s(h),t(h)\}=\{i,j\}\right\}\quad \textit{ for } i\neq j.
$$
Notice that the assumptions \cref{equ:simplylaced} and \cref{eq:nofixedpoint} on the quiver are equivalent to the assumptions on the Cartan matrix in \cref{sec:Yangian}.
Let $\mathfrak{g}_Q$ denote the corresponding Kac--Moody algebra, with $\Lambda_i$ (resp. $\alpha_i^\vee$) being the fundamental coweights (resp. simple coroots). Fix dimension vectors $(v_i)_{i\in Q_0}$ and $(w_i)_{i\in Q_0}$, such that $v_i=v_{\tau i}$. Then we have the quantized Coulomb branch algebra $$\mathcal{A}_\hbar := 
H^{(G_{\calO}\times F_{\calO})\rtimes\mathbb{C}_\hbar^\times}_*(\calR_{G,\bfN})$$ 
associated to the following data as in Section \ref{sec:mono}:
$$G:=G_V^{\tau} \simeq \prod_{i\in Q_0^+}GL(V_i),\qquad \bfN = E_V^\tau \oplus L_{W,V},\qquad 
F=G_W. $$
By \cref{eq:embeddiff}, we have the embedding 
\begin{equation}\label{equ:embedcou}
    z^*\circ(\iota_*)^{-1}:\mathcal{A}_\hbar\hookrightarrow \Diff_\hbar(T_V^\tau)\otimes H_{G_W}^*(\pt).
\end{equation}

Let 
$$\mu=\sum_{j\in Q_0} (w_j+w_{\tau j})\Lambda_j-\sum_{h\in Q_1^\tau}(\Lambda_{s(h)}+\Lambda_{t(h)})-\sum_{j\in Q_0}v_j\alpha_j^\vee\in P^\vee,$$
which is clearly $\tau$-invariant.
For $i\in Q_0$, we define the following elements in $\Diff_\hbar(T_V^\tau)\otimes H_{G_W}^*(\pt)$:
\begin{align}\label{equ:Biz}
    B_i(u) =&\sum_{r=1}^{v_i}\frac{1}{-u-x_{i,r}-\tfrac{\hbar}{2}} \prod_{h\in Q_1^\tau\atop s(h)=i}V_{\tau i,r}(x_{i,r}+\tfrac{\hbar}{2})\prod_{h\notin Q_1^\tau\atop s(h)=i}V_{t(h)}(x_{i,r}+\tfrac{\hbar}{2})\frac{W_{\tau i}(-x_{i,r}-\tfrac{\hbar}{2})}{V_{i,r}(x_{i,r})}d_{i,r}\notag\\ =&\sum_{r=1}^{v_i}\frac{1}{-u-x_{i,r}-\tfrac{\hbar}{2}} \frac{\prod_{h\in Q_1\atop s(h)=i}V_{t(h)}(x_{i,r}+\tfrac{\hbar}{2})}{\prod_{h\in Q_1^\tau\atop s(h)=i}(2x_{i,r}+\tfrac{\hbar}{2})}\frac{W_{\tau i}(-x_{i,r}-\tfrac{\hbar}{2})}{V_{i,r}(x_{i,r})}d_{i,r},
\end{align}
and
\begin{equation}\label{equ:Hu}
    H_i(u) := (-1)^{v_i-1}(2u)^{c_{i,\tau i}}(-1)^{\delta_{i\to \tau(i)}}\frac{W_i(-u)W_{\tau i}(u)}{V_i(-u+\frac{\hbar}{2})V_i(-u-\tfrac{\hbar}{2})}\prod_{h\in Q_1\atop s(h)=i}V_{t(h)}(-u)\prod_{h\in Q_1\atop s(h)=\tau i}V_{t(h)}(u),
\end{equation}
where $\delta_{i\to \tau(i)}$ is $1$ if there is an $h\in Q_1$ such that $s(h)=i$ and $t(h)=\tau(i)$, and is $0$ otherwise. 
Expanding $B_i(u)$ and $H_i(u)$ into a Laurent series in $u^{-1}$, we let $B_{i,r}$ (resp. $H_{i,r}$) denote the coefficient of $u^{-r-1}$ of $B_i(u)$ (resp. $H_i(u)$). 
\begin{lem}\label{lem:Hiz}
    The following holds:
    \begin{enumerate}
        \item $H_{\tau i}(u) = H_i(-u)$.
        \item $H_{i,r}\in H_{G_V^\tau\times G_W\times \mathbb{C}^\times_\hbar}^*(\pt)$
        \item $H_{i,r}=0$ if $r<-\langle \alpha_i,\mu\rangle -1$, and $H_{i,-\langle \alpha_i,\mu\rangle -1}=2^{c_{i,\tau i}}(-1)^{v_i-1+\delta_{i\rightarrow \tau i}+w_i+\sum_{h\in Q_1, s(h)=i}v_{t(h)}}$.
    \end{enumerate}
\end{lem}
\begin{proof}
    The first one follows from the fact $x_{i,j}=-x_{\tau i, j}$, while the second one follows from the fact that the coefficients of $V_i(u)$ (resp. $W_i(u)$) are in $H_{G_V^\tau}^*(\pt)$ (resp. $H_{G_W}^*(\pt)$). By definition, the highest degree of $u$ in $H_i(u)$ is
    \[c_{i,\tau i}+w_i+w_{\tau i}-2v_i+\sum_{h\in Q_1\atop s(h)=i}v_{t(h)}+\sum_{h\in Q_1\atop s(h)=\tau i}v_{t(h)}=\langle \alpha_i,\mu\rangle.\]
    This concludes the lemma.
\end{proof}

From now on, we fix the parameters 
\[\zeta_i:=\hbar^{-1}2^{c_{i,\tau i}}(-1)^{v_i-1+\delta_{i\rightarrow \tau i}+w_i+\sum_{h\in Q_1, s(h)=i}v_{t(h)}},\]
and consider the associated shifted twisted Yangian $\mathbf{Y}^{\tau}_{\mu}(\mathfrak{g}_Q)$. Then we have the following GKLO-type representation for the shifted twisted Yangian, generalizing the results from \cite{GKLO05,KWWY14,BFN19}.
\begin{theo}\label{thm:gklo}
There exists a unique $H_{G_W}^*(\pt)(\hbar)$-algebra homomorphism 
\begin{equation}
\label{gklo}
    \psi\colon\mathbf{Y}^{\tau}_{\mu}(\mathfrak{g}_Q)
\otimes H_{G_W}^*(\pt)
\longrightarrow 
\Diff_\hbar(T_V^\tau)\otimes H_{G_W}^*(\pt)
\end{equation}
sending $h_i(z)$ to $H_i(z)$ and $b_i(z)$ to $B_i(z)$. 
\end{theo}

To prove the theorem, we need to check that the operators $B_i(z)$ and $H_i(z)$ satisfy all the relations in $\mathbf{Y}^{\tau}_{\mu}(\mathfrak{g}_Q)
\otimes H_{G_W}^*(\pt)$. We show this in \cref{sec:relations} and \cref{sec:iserre} below. 
Many of the relations can be verified using techniques similar to those in \cite[Appendix B]{BFN19}, and we will only sketch the arguments for those cases. In fact, the verification can be slightly simplified using a rational function identity; see Lemma~\ref{lem:partialfrac}.   

As discussed in \cref{rem:iSerre}, the relation \eqref{eq:iSerre} presents additional subtleties. We revisit and reformulate it carefully in \cref{sec:gftoiSerre}. Unlike the other relations, \eqref{eq:iSerre} involves a complicated summation, and there does not seem to be a generating function formulation for this. To address this difficulty, we establish an analogue of \cite[Section 3.4]{LZ24} on the Coulomb branch side. This approach enables us to verify the relation \eqref{eq:iSerre} only for the special case \( k_1 = k_2 = 0 \), where a generating function formulation is available; see \cref{equ:iSerreitaui}.

Recall the embedding \cref{equ:embedcou}. As an immediate consequence of the above theorem, we get our main result. 
\begin{theo}\label{thm:main}
   The homomorphism $\psi$ in \cref{thm:gklo} factors through a unique $H_{G_W}^*(\pt)[\hbar,\hbar^{-1}]$-algebra homomorphism $\Psi$:
   \[\xymatrix{\mathbf{Y}^{\tau}_{\mu}(\mathfrak{g}_Q)\otimes H_{G_W}^*(\pt) \ar[rr]^-\Psi \ar[rd]^-\psi & &\mathcal{A}_\hbar[\hbar^{-1}] \ar@{^{(}->}[ld]^-{z^*\circ(\iota_*)^{-1}}\\
   &\Diff_\hbar(T_V^\tau)\otimes H_{G_W}^*(\pt),}\] 
    such that for any $i\in Q_0$, $\Psi(h_{i,r})=\hbar^{-1}H_{i,r}$;
    for $i\in Q_0^+$,
    \[\Psi(b_{i,r})= \hbar^{-1}(-1)^{1+\sum_{h\notin Q_1^\tau\atop s(h)=i} v_{t(h)}+\sum_{h\in Q_1^\tau \atop s(h)=i}(v_i-1)} (-c_1(\mathcal{Q}_i)-\tfrac{\hbar}{2})^r\cap [R_{\epsilon_{i,1}}], \]
    and for $i\in Q_0^-$,
    \[\Psi(b_{i,r})= \hbar^{-1}(-1)^{1+\sum_{h\notin Q_1^\tau\atop s(h)=i} v_{t(h)}+\sum_{h\in Q_1^\tau \atop s(h)=i}(v_i-1)} (c_1(\mathcal{S}_{\tau i})+\tfrac{\hbar}{2})^r\cap [R_{-\epsilon_{\tau i,v_i}}]. \]
\end{theo}
\begin{proof}
    This follows directly from \cref{prop:fmono}, \cref{prop:fmono2}, and \cref{thm:gklo}.
\end{proof}
\begin{rem}
    By \cref{lem:Hiz}(2), the image of $h_{i,r}$ lies in the Gelfand--Tsetlin subalgebra $H_{G_V^\tau\times G_W\times \mathbb{C}^\times_\hbar}^*(\pt)$ of $\mathcal{A}_\hbar[\hbar^{-1}]$.  
\end{rem}

\section{Relations (I)}\label{sec:relations}
In this section, we start to prove Theorem \ref{thm:gklo} by checking that the operators $B_i(z)$ and $H_i(z)$ satisfy the relations in the shifted twisted Yangian. 

The relation $\left[H_i(u), H_j(v)\right]=0$ is obvious, while $H_{\tau i}(u)=H_i(-u)$ is proved in \cref{lem:Hiz}(1). Before checking Relation \eqref{eq:hbcomm}, let us introduce some notations. For $i\in Q_0$ and $1\leq r\leq v_i$, let 
\begin{align*}
    y_{i,r}:=&\prod_{h\in Q_1^\tau \atop s(h)=i}V_{\tau i,r}(x_{i,r}+\tfrac{\hbar}{2})\prod_{h\notin Q_1^\tau \atop s(h)=i}V_{t(h)}(x_{i,r}+\tfrac{\hbar}{2})\cdot \frac{W_{\tau i}(-x_{i,r}-\tfrac{\hbar}{2})}{V_{i,r}(x_{i,r})}d_{i,r}\\
    =&\frac{\prod_{h\in Q_1\atop s(h)=i}V_{t(h)}(x_{i,r}+\tfrac{\hbar}{2})}{\prod_{h\in Q_1^\tau\atop s(h)=i}(2x_{i,r}+\tfrac{\hbar}{2})}\frac{W_{\tau i}(-x_{i,r}-\tfrac{\hbar}{2})}{V_{i,r}(x_{i,r})}d_{i,r}. 
\end{align*}
Then
\begin{align*}
    y_{\tau i,r}=&\prod_{h\in Q_1^\tau\atop t(h)=i}V_{i,r}(-x_{i,r}+\tfrac{\hbar}{2})\prod_{h\notin Q_1^\tau\atop s(h)=\tau i}V_{t(h)}(-x_{i,r}+\tfrac{\hbar}{2})\cdot \frac{W_i(x_{i,r}-\tfrac{\hbar}{2})}{(-1)^{v_i-1}V_{i,r}(x_{i,r})}d_{i,r}^{-1}\\
=&\frac{\prod_{h\in Q_1\atop s(h)=\tau i}V_{t(h)}(-x_{i,r}+\tfrac{\hbar}{2})}{\prod_{h\in Q_1^\tau\atop t(h)=i}(-2x_{i,r}+\tfrac{\hbar}{2})}\frac{W_i(x_{i,r}-\tfrac{\hbar}{2})}{(-1)^{v_i-1}V_{i,r}(x_{i,r})}d_{i,r}^{-1},
\end{align*}
and
\[B_i(z) =\sum_{r=1}^{v_i}\frac{1}{-z-x_{i,r}-\tfrac{\hbar}{2}} y_{i,r}.\]

It is direct to get the following commutation relations. First of all,
\[y_{i,r}x_{j,s}=(x_{j,s}+\delta_{i,j}\delta_{r,s}\hbar-\delta_{\tau i,j}\delta_{r,s}\hbar)y_{i,r}.\] For $1\leq r\neq s\leq v_i$,
\begin{equation}\label{equ:yiyi}
	y_{i,s}y_{i,r}=\frac{x_{i,r}-x_{i,s}+\hbar}{x_{i,r}-x_{i,s}-\hbar}y_{i,r}y_{i,s},
\end{equation}
and
\begin{equation}\label{equ:ytauiyi}
	y_{\tau i,s}y_{i,r}=\Big(\frac{x_{i,r}-\tfrac{\hbar}{2}+x_{i,s}}{x_{i,r}+\tfrac{\hbar}{2}+x_{i,s}}\Big)^{-c_{i,\tau i}}y_{i,r}y_{\tau i,s}.
\end{equation}

If $c_{ij}=-1$, for $1\leq r\leq v_i$ and $1\leq s\leq v_j$, then
\begin{equation}\label{equ:yjyi}
	y_{j,s}y_{i,r}=\Big(\frac{x_{i,r}-\tfrac{\hbar}{2}-x_{j,s}}{x_{i,r}+\tfrac{\hbar}{2}-x_{j,s}}\Big)^{-c_{ij}}y_{i,r}y_{j,s}.
\end{equation}

\subsection{Relation \eqref{eq:hbcomm}: commutativity of $h_{i,r}$ and $b_{j,s}$}
In this section, we check the relation \eqref{eq:hbcomm(z)}
\begin{align}\label{equ:hibj}
&
\left(u^2-v^2-\frac{c_{i j} c_{\tau i, j}}{4} \hbar^2\right)\left[h_i(u), b_j(v)\right]-\bigg(\frac{c_{i j}-c_{\tau i, j}}{2} \hbar u +\frac{c_{i j}+c_{\tau i, j}}{2} \hbar v\bigg)\left\{h_i(u), b_j(v)\right\}\\
&\notag
+\hbar\left[h_i(u), b_{j, 1}\right] 
+\hbar v\left[h_i(u), b_{j, 0}\right]+\frac{c_{i j}+c_{\tau i, j}}{2} \hbar^2\left\{h_i(u), b_{j, 0}\right\}=0
\end{align}
This is divided into two cases: $j\in \{i,\tau i\}$ and $j\notin \{i,\tau i\}$.

\subsubsection{The case $j\in \{i,\tau i\}$}
Assume $j=i$. 
From definition, it is easy to get
\begin{align*}
    B_i(v)H_i(u)
    =&\sum_{r=1}^{v_i}\frac{1}{-v-x_{i,r}-\tfrac{\hbar}{2}} \frac{u-\tfrac{\hbar}{2}+x_{i,r}}{u+\tfrac{3\hbar}{2}+x_{i,r}}\Big(\frac{u-x_{i,r}-\hbar}{u-x_{i,r}}\Big)^{-c_{i,\tau i}}H_i(u)y_{i,r}.
\end{align*}
Hence, the coefficient of $H_i(u)y_{i,r}$ in the left hand side of \eqref{equ:hibj} is $0$ because of the following identity (notice that $c_{i,\tau i}=0$ or $-1$):
\begin{align*}
    &\big(u^2-v^2-\frac{ c_{\tau i, i}}{2}\hbar^2 \big)\frac{1}{-v-x_{i,r}-\tfrac{\hbar}{2}}\bigg(1-\frac{u-\tfrac{\hbar}{2}+x_{i,r}}{u+\tfrac{3\hbar}{2}+x_{i,r}}\Big(\frac{u-x_{i,r}-\hbar}{u-x_{i,r}}\Big)^{-c_{i,\tau i}}\bigg)\\
    &-\Big(\frac{2-c_{\tau i, i}}{2}u+\frac{2+c_{\tau i, i}}{2}v\Big)\frac{1}{-v-x_{i,r}-\tfrac{\hbar}{2}}\hbar \bigg(1+\frac{u-\tfrac{\hbar}{2}+x_{i,r}}{u+\tfrac{3\hbar}{2}+x_{i,r}}\Big(\frac{u-x_{i,r}-\hbar}{u-x_{i,r}}\Big)^{-c_{i,\tau i}}\bigg)\\
    &+(x_{i,r}+\tfrac{\hbar}{2})\bigg(1-\frac{u-\tfrac{\hbar}{2}+x_{i,r}}{u+\tfrac{3\hbar}{2}+x_{i,r}}\Big(\frac{u-x_{i,r}-\hbar}{u-x_{i,r}}\Big)^{-c_{i,\tau i}}\bigg)\\
    &-v\bigg(1-\frac{u-\tfrac{\hbar}{2}+x_{i,r}}{u+\tfrac{3\hbar}{2}+x_{i,r}}\Big(\frac{u-x_{i,r}-\hbar}{u-x_{i,r}}\Big)^{-c_{i,\tau i}}\bigg)\\
    &-\frac{2+c_{\tau i,i}}{2}\hbar \bigg(1+\frac{u-\tfrac{\hbar}{2}+x_{i,r}}{u+\tfrac{3\hbar}{2}+x_{i,r}}\Big(\frac{u-x_{i,r}-\hbar}{u-x_{i,r}}\Big)^{-c_{i,\tau i}}\bigg) = 0.
\end{align*}
Hence, \eqref{equ:hibj} holds in this case. 
The case $j=\tau i$ can be checked similarly.

\subsubsection{The case $j\notin \{i,\tau i\}$}
Similar to the above,
\begin{align*}
    B_{j}(v)H_i(u)
    =&\sum_{r=1}^{v_j}\frac{1}{-v-x_{j,r}-\tfrac{\hbar}{2}}
    \Big(\frac{u-x_{j,r}-\hbar}{u-x_{j,r}}\Big)^{-c_{\tau i, j}}\Big(\frac{u+x_{j,r}+\hbar}{u+x_{j,r}}\Big)^{-c_{i,j}} H_i(u)y_{j,r}.
\end{align*}
\eqref{equ:hibj} holds because of the following identity (notice that $c_{i,j}, c_{\tau i,j}=0$ or $-1$):
\begin{align*}
    &\big(u^2-v^2-\frac{ c_{ij}c_{\tau i,j}}{4}\hbar^2 \big)\frac{1}{-v-x_{j,r}-\tfrac{\hbar}{2}}\bigg(1-\Big(\dfrac{u-x_{j,r}-\hbar}{u-x_{j,r}}\Big)^{-c_{\tau i, j}}\Big(\frac{u+x_{j,r}+\hbar}{u+x_{j,r}}\bigg)^{-c_{i,j}}\Big)\\
    &-\Big(\frac{c_{ij}-c_{\tau i, j}}{2}\hbar u+\frac{c_{ij}+c_{\tau i, j}}{2}\hbar v\Big)\frac{1}{-v-x_{j,r}-\tfrac{\hbar}{2}} \bigg(1+\Big(\frac{u-x_{j,r}-\hbar}{u-x_{j,r}}\Big)^{-c_{\tau i, j}}\Big(\frac{u+x_{j,r}+\hbar}{u+x_{j,r}}\Big)^{-c_{i,j}}\bigg)\\
    &+(x_{j,r}+\tfrac{\hbar}{2})\bigg(1-\Big(\frac{u-x_{j,r}-\hbar}{u-x_{j,r}}\Big)^{-c_{\tau i, j}}\Big(\frac{u+x_{j,r}+\hbar}{u+x_{j,r}}\Big)^{-c_{i,j}}\bigg)\\
    &-v\bigg(1-\Big(\frac{u-x_{j,r}-\hbar}{u-x_{j,r}}\Big)^{-c_{\tau i, j}}\Big(\frac{u+x_{j,r}+\hbar}{u+x_{j,r}}\Big)^{-c_{i,j}}\bigg)\\
    &-\frac{c_{ij}+c_{\tau i,j}}{2}\hbar \bigg(1+\Big(\frac{u-x_{j,r}-\hbar}{u-x_{j,r}}\Big)^{-c_{\tau i, j}}\Big(\frac{u+x_{j,r}+\hbar}{u+x_{j,r}}\Big)^{-c_{i,j}}\bigg) = 0.
\end{align*}

\subsection{Some identities}
Now we turn to the other relations. We need some combinatorial identities.
For a Laurent series $f(z)=\sum_{i\in \mathbb{Z}} a_iz^{i}$ in $z^{-1}$ ,
we denote the truncation 
\[\big(f(z)\big)^{\circ}
:=\sum_{i< 0}a_iz^i.\] 
If $\tilde{f}(z):=f(z)/z$, then
\begin{equation}\label{equ:trundivbyz}
    u\big(\tilde{f}(u)\big)^\circ+v\big(\tilde{f}(-v)\big)^\circ
    =u\sum_{i< 1}a_iu^{i-1}-(-v)\sum_{i< 1}a_i(-v)^{i-1}=\big(f(u)\big)^{\circ}-\big(f(-v)\big)^{\circ}.
\end{equation}

\begin{lem}\label{lem:partialfrac}
    Assume $f(z)=\dfrac{g(z)}{\prod_{i=1}^n (z-z_i)}$, where $z_1,\ldots,z_n$ are distinct and $g(z)\in \mathbb{C}[z]$. Then
\[(f(z))^\circ=\sum_{i=1}^n\frac{1}{z-z_i}\frac{g(z_i)}{\prod_{j\neq i}(z_i-z_j)}.\]
\end{lem}
\begin{proof}
By partial fractions, we can write $$f(z)=\sum_{i=1}^n\frac{b_i}{z-z_i}+h(z),$$
for some polynomial $h(z)\in \mathbb{C}[z]$ and the coefficients $b_i$ can be determined by the residue
\[b_i=\operatorname{Res}_{z=z_i}f(z)dz=\frac{g(z_i)}{\prod_{j\neq i}(z_i-z_j)}.
\qedhere\]
\end{proof}

Recall the operator $H_i(u)$ from \cref{equ:Hu}.
\begin{cor}\label{cor:cequal0H}
    If $c_{i,\tau i}=0$,
    \begin{align*}
    \hbar \cdot (H_i(u))^\circ=\sum_{r=1}^{v_i} \frac{1}{u+x_{i,r}-\tfrac{\hbar}{2}}y_{\tau i,r}y_{i,r}-\sum_{r=1}^{v_i}\frac{1}{u+x_{i,r}+\tfrac{\hbar}{2}}y_{i,r}y_{\tau i,r}.
    \end{align*}
\end{cor}
\begin{proof}
Since $c_{i,\tau i}=0$, there is no $h\in Q_1^\tau$ with $s(h)=i$. Hence,
\[y_{i,r}=\prod_{h\notin Q_1^\tau\atop s(h)=i}V_{t(h)}(x_{i,r}+\tfrac{\hbar}{2})\frac{W_{\tau i}(-x_{i,r}-\tfrac{\hbar}{2})}{V_{i,r}(x_{i,r})}d_{i,r}\\\]
Therefore,
\begin{align*}
    &\sum_{r=1}^{v_i} \frac{1}{u+x_{i,r}-\tfrac{\hbar}{2}}y_{\tau i,r}y_{i,r}-\sum_{r=1}^{v_i}\frac{1}{u+x_{i,r}+\tfrac{\hbar}{2}}y_{i,r}y_{\tau i,r}\\
    =\,&\sum_{r=1}^{v_i} \frac{1}{u+x_{i,r}-\tfrac{\hbar}{2}}\prod_{h\notin Q_1^\tau\atop s(h)=\tau i}V_{t(h)}(-x_{ i,r}+\tfrac{\hbar}{2})\frac{W_{ i}(x_{i,r}-\tfrac{\hbar}{2})}{(-1)^{v_i-1}V_{i,r}(x_{i,r})}\prod_{h\notin Q_1^\tau\atop s(h)=i}V_{t(h)}(x_{i,r}-\tfrac{\hbar}{2})\frac{W_{\tau i}(-x_{i,r}+\tfrac{\hbar}{2})}{V_{i,r}(x_{i,r}-\hbar)}\\
    &-\sum_{r=1}^{v_i}\frac{1}{u+x_{i,r}+\tfrac{\hbar}{2}}\prod_{h\notin Q_1^\tau\atop s(h)=i}V_{t(h)}(x_{i,r}+\tfrac{\hbar}{2})\frac{W_{\tau i}(-x_{i,r}-\tfrac{\hbar}{2})}{V_{i,r}(x_{i,r})}\prod_{h\notin Q_1^\tau\atop s(h)=\tau i}V_{t(h)}(-x_{ i,r}-\tfrac{\hbar}{2})\frac{W_{ i}(x_{i,r}+\tfrac{\hbar}{2})}{(-1)^{v_i-1}V_{i,r}(x_{i,r}+\hbar)}\\
    =\,&\hbar\cdot (H_i(u))^\circ.
\end{align*}
Here the last equality follows from Lemma \ref{lem:partialfrac}.
\end{proof}

\begin{cor}\label{cor:uH}
    For any $i\in Q_0$,
    \begin{align*}
    2\hbar u(H_i(u))^\circ+2\hbar v(H_i(-v))^\circ
    & =\sum_{r=1}^{v_i}\bigg(\frac{2x_{i,r}+\hbar+\frac{c_{i,\tau i}}{2}\hbar}{u+x_{i,r}+\tfrac{\hbar}{2}}y_{i,r}y_{\tau i, r}-\frac{2x_{i,r}-\hbar-\frac{c_{i,\tau i}}{2}\hbar}{u+x_{i,r}-\tfrac{\hbar}{2}}y_{\tau i,r}y_{i, r}\bigg)\\
    &\quad -\sum_{r=1}^{v_i}\bigg(\frac{2x_{i,r}+\hbar+\frac{c_{i,\tau i}}{2}\hbar}{-v+x_{i,r}+\tfrac{\hbar}{2}}y_{i,r}y_{\tau i, r}-\frac{2x_{i,r}-\hbar-\frac{c_{i,\tau i}}{2}\hbar}{-v+x_{i,r}-\tfrac{\hbar}{2}}y_{\tau i,r}y_{i, r}\bigg)
\end{align*}
\end{cor}
\begin{proof}
    If $c_{i,\tau i}=0$, then by Lemma \ref{lem:partialfrac} and Corollary \ref{cor:cequal0H}, we get
    \begin{align*}
    (2\hbar u\cdot H_i(u))^\circ=-\sum_{r=1}^{v_i} \frac{2x_{i,r}-\hbar}{u+x_{i,r}-\tfrac{\hbar}{2}}y_{\tau i,r}y_{i,r}+\sum_{r=1}^{v_i}\frac{2x_{i,r}+\hbar}{u+x_{i,r}+\tfrac{\hbar}{2}}y_{i,r}y_{\tau i,r}.
    \end{align*}
    Hence,
    \begin{align*}
        & \sum_{r=1}^{v_i}\bigg(\frac{2x_{i,r}+\hbar}{u+x_{i,r}+\tfrac{\hbar}{2}}y_{i,r}y_{\tau i, r}-\frac{2x_{i,r}-\hbar}{u+x_{i,r}-\tfrac{\hbar}{2}}y_{\tau i,r}y_{i, r}\bigg)\\
        &-\sum_{r=1}^{v_i}\bigg(\frac{2x_{i,r}+\hbar}{-v+x_{i,r}+\tfrac{\hbar}{2}}y_{i,r}y_{\tau i, r}-\frac{2x_{i,r}-\hbar}{-v+x_{i,r}-\tfrac{\hbar}{2}}y_{\tau i,r}y_{i, r}\bigg)\\
        =\,&(2\hbar u\cdot H_i(u))^\circ - (2\hbar\cdot (-v) H_i(-v))^\circ
        =2\hbar u (H_i(u))^\circ + 2\hbar v  (H_i(-v))^\circ.
    \end{align*}
    where the last equality follows from \eqref{equ:trundivbyz}.

    Now let us prove the case $c_{i,\tau i}=-1$. Since $H_{\tau i}(u)= H_i(-u)$, we can assume that $i$ is the source of the unique edge connecting $i$  and $\tau i$. Similar as in the proof of Corollary \ref{cor:cequal0H},
    \begin{align*}
        &\sum_{r=1}^{v_i}\frac{2x_{i,r}+\hbar+\frac{c_{i,\tau i}}{2}\hbar}{u+x_{i,r}+\tfrac{\hbar}{2}}y_{i,r}y_{\tau i, r}-\sum_{r=1}^{v_i}\frac{2x_{i,r}-\hbar-\frac{c_{i,\tau i}}{2}\hbar}{u+x_{i,r}-\tfrac{\hbar}{2}}y_{\tau i,r}y_{i, r}\\
        =\,&\sum_{r=1}^{v_i}\frac{1}{u+x_{i,r}+\tfrac{\hbar}{2}}\prod_{h\in Q_1\atop s(h)=i}V_{t(h)}(x_{i,r}+\tfrac{\hbar}{2})\frac{W_{\tau i}(-x_{i,r}-\tfrac{\hbar}{2})}{V_{i,r}(x_{i,r})}\\
        &\cdot \prod_{h\in Q_1\atop s(h)=\tau i}V_{t(h)}(-x_{i,r}-\tfrac{\hbar}{2})\frac{W_i(x_{i,r}+\tfrac{\hbar}{2})}{(-1)^{v_i-1}V_{i,r}(x_{i,r}+\hbar)}\\
        &-\sum_{r=1}^{v_i}\frac{1}{u+x_{i,r}-\tfrac{\hbar}{2}}\prod_{h\in Q_1\atop s(h)=\tau i}V_{t(h)}(-x_{i,r}+\tfrac{\hbar}{2})\frac{W_i(x_{i,r}-\tfrac{\hbar}{2})}{(-1)^{v_i-1}V_{i,r}(x_{i,r})}\\
        &\cdot \prod_{h\in Q_1\atop s(h)=i}V_{t(h)}(x_{i,r}-\tfrac{\hbar}{2})\frac{W_{\tau i}(-x_{i,r}+\tfrac{\hbar}{2})}{V_{i,r}(x_{i,r}-\hbar)} 
        = (2\hbar uH_i(u))^\circ,
    \end{align*}
    where the last equality follows from Lemma \ref{lem:partialfrac}.
Hence,
\begin{align*}
    &\sum_{r=1}^{v_i}\bigg(\frac{2x_{i,r}+\hbar+\frac{c_{i,\tau i}}{2}\hbar}{u+x_{i,r}+\tfrac{\hbar}{2}}y_{i,r}y_{\tau i, r}-\frac{2x_{i,r}-\hbar-\frac{c_{i,\tau i}}{2}\hbar}{u+x_{i,r}-\tfrac{\hbar}{2}}y_{\tau i,r}y_{i, r}\bigg)\\
    &-\sum_{r=1}^{v_i}\bigg(\frac{2x_{i,r}+\hbar+\frac{c_{i,\tau i}}{2}\hbar}{-v+x_{i,r}+\tfrac{\hbar}{2}}y_{i,r}y_{\tau i, r}-\frac{2x_{i,r}-\hbar-\frac{c_{i,\tau i}}{2}\hbar}{-v+x_{i,r}-\tfrac{\hbar}{2}}y_{\tau i,r}y_{i, r}\bigg)\\
    =\,&(2\hbar uH_i(u))^\circ-(2\hbar (-v)H_i(-v))^\circ
    =2\hbar u (H_i(u))^\circ + 2\hbar v  (H_i(-v))^\circ,
\end{align*}
where the last equality follows from \eqref{equ:trundivbyz}.    
\end{proof}

\subsection{Relation \eqref{eq:bcomm}: commutativity of $b_{i,r}$'s}
In this section, we check the relation \eqref{eq:bcomm(z)}
\begin{align}\label{equ:bbij}
    (u-v)\left[b_i(u), b_j(v)\right]-&\frac{c_{i j}}{2} \hbar\left\{b_i(u), b_j(v)\right\}-\hbar\left(\left[b_{i, 0}, b_j(v)\right]-\left[b_i(u), b_{j, 0}\right]\right) \\\notag&\qquad =-\delta_{\tau i, j} \hbar\left(\frac{2 u}{u+v} h_i^{\circ}(u)+\frac{2 v}{u+v} h_j^{\circ}(v)\right).
\end{align}
Recall that by our convention, $\hbar b_{i,m}$ maps to $B_{i,m}$, the coefficient of $u^{-m-1}$ in $B_i(u)$.
If $j\neq \tau i$, and $c_{ij}=0$, it follows from the relation $[B_i(u),B_j(v)]=0$ checked in Section \ref{sec:BBijc=0} below. Hence, we are left with the cases $j= \tau i$ or $c_{ij}\neq 0$, which is further decomposed into the following cases
\begin{itemize}
    \item $j\neq \tau i$, $c_{ij}= 2$; 
    \item $j\neq \tau i$, $c_{ij}= -1$;
    \item $j=\tau i$. 
\end{itemize}

\subsubsection{The case $j\neq \tau i$, $i=j$}
By definition,
\begin{align*}
    B_i(u)B_i(v)
   =\,&\sum_{r=1}^{v_i}\frac{1}{-u-x_{i,r}-\tfrac{\hbar}{2}} \frac{1}{-v-x_{i,r}-\tfrac{3\hbar}{2}} y_{i,r}y_{i,r}\\
    &+\sum_{1\leq r\neq s\leq v_i}\frac{1}{-u-x_{i,r}-\tfrac{\hbar}{2}} \frac{1}{-v-x_{i,s}-\tfrac{\hbar}{2}} y_{i,r}y_{i,s},
\end{align*}
and
\begin{align*}
    B_i(v)B_i(u)
    =\,&\sum_{r=1}^{v_i}\frac{1}{-u-x_{i,r}-\tfrac{3\hbar}{2}} \frac{1}{-v-x_{i,r}-\tfrac{\hbar}{2}} y_{i,r}y_{i,r}\\
    &+\sum_{1\leq r\neq s\leq v_i}\frac{1}{-u-x_{i,r}-\tfrac{\hbar}{2}} \frac{1}{-v-x_{i,s}-\tfrac{\hbar}{2}} y_{i,s}y_{i,r}.
\end{align*}

Therefore, for $r\neq s$, the coefficient of $y_{i,r}y_{i,s}$ in the left hand side of \eqref{equ:bbij} is
\begin{align*}
    &(u-v-\hbar)\frac{1}{-u-x_{i,r}-\tfrac{\hbar}{2}} \frac{1}{-v-x_{i,s}-\tfrac{\hbar}{2}}\\
    &-(u-v+\hbar)\frac{1}{-u-x_{i,r}-\tfrac{\hbar}{2}} \frac{1}{-v-x_{i,s}-\tfrac{\hbar}{2}}\frac{x_{i,r}-x_{i,s}+\hbar}{x_{i,r}-x_{i,s}-\hbar}\\
    &+\frac{1}{-v-x_{i,s}-\tfrac{\hbar}{2}}-\frac{1}{-u-x_{i,r}-\tfrac{\hbar}{2}}\\
    &-\frac{1}{-v-x_{i,s}-\tfrac{\hbar}{2}}\frac{x_{i,r}-x_{i,s}+\hbar}{x_{i,r}-x_{i,s}-\hbar}+\frac{1}{-u-x_{i,r}-\tfrac{\hbar}{2}}\frac{x_{i,r}-x_{i,s}+\hbar}{x_{i,r}-x_{i,s}-\hbar} = 0. 
\end{align*}
Here we have used \eqref{equ:yiyi}. Similarly, the coefficient of $y_{i,r}y_{i,r}$ in the left hand side of \eqref{equ:bbij} is
\begin{align*}
    &(u-v-\hbar)\frac{1}{-u-x_{i,r}-\tfrac{\hbar}{2}} \frac{1}{-v-x_{i,r}-\tfrac{3\hbar}{2}}-(u-v+\hbar)\frac{1}{-u-x_{i,r}-\tfrac{3\hbar}{2}} \frac{1}{-v-x_{i,r}-\tfrac{\hbar}{2}} \\
    &+\frac{1}{-v-x_{i,r}-\tfrac{3\hbar}{2}}-\frac{1}{-v-x_{i,r}-\tfrac{\hbar}{2}}-\frac{1}{-u-x_{i,r}-\tfrac{\hbar}{2}}+\frac{1}{-u-x_{i,r}-\tfrac{3\hbar}{2}} = 0. 
\end{align*}
Therefore, \eqref{equ:bbij} holds in this case. 

\subsubsection{The case $j\neq \tau i$, $c_{ij}=-1$}
Since the left hand side of \eqref{equ:bbij} is symmetric about $i$ and $j$, we can assume that $i$ is the source of the edge connecting $i$ and $j$.
By definition,
\[B_i(u)B_j(v)=\sum_{r=1}^{v_i}\sum_{s=1}^{v_j}\frac{1}{-u-x_{i,r}-\tfrac{\hbar}{2}} \frac{1}{-v-x_{j,s}-\tfrac{\hbar}{2}} y_{i,r}y_{j,s},\]
and by \eqref{equ:yjyi},
\[B_j(v)B_i(u)=\sum_{r=1}^{v_i}\sum_{s=1}^{v_j}\frac{1}{-u-x_{i,r}-\tfrac{\hbar}{2}} \frac{1}{-v-x_{j,s}-\tfrac{\hbar}{2}} \frac{x_{i,r}-\tfrac{\hbar}{2}-x_{j,s}}{x_{i,r}+\tfrac{\hbar}{2}-x_{j,s}}y_{i,r}y_{j,s}.\]
Therefore, the coefficient of $y_{i,r}y_{j,s}$ in the left hand side of \eqref{equ:bbij} is
\begin{align*}
    &(u-v+\tfrac{\hbar}{2})\frac{1}{-u-x_{i,r}-\tfrac{\hbar}{2}} \frac{1}{-v-x_{j,s}-\tfrac{\hbar}{2}}\\
    &-(u-v-\tfrac{\hbar}{2})\frac{1}{-u-x_{i,r}-\tfrac{\hbar}{2}} \frac{1}{-v-x_{j,s}-\tfrac{\hbar}{2}} \frac{x_{i,r}-\tfrac{\hbar}{2}-x_{j,s}}{x_{i,r}+\tfrac{\hbar}{2}-x_{j,s}}\\ 
    &+\frac{1}{-v-x_{j,s}-\tfrac{\hbar}{2}}-\frac{1}{-v-x_{j,s}-\tfrac{\hbar}{2}} \frac{x_{i,r}-\tfrac{\hbar}{2}-x_{j,s}}{x_{i,r}+\tfrac{\hbar}{2}-x_{j,s}}\\
    &-\frac{1}{-u-x_{i,r}-\tfrac{\hbar}{2}}+\frac{1}{-u-x_{i,r}-\tfrac{\hbar}{2}}  \frac{x_{i,r}-\tfrac{\hbar}{2}-x_{j,s}}{x_{i,r}+\tfrac{\hbar}{2}-x_{j,s}} = 0. 
\end{align*}
Hence, \eqref{equ:bbij} holds.

\subsubsection{The case $j=\tau i$}\label{sec:BBitaui}
We need to show
\begin{align}\label{equ:BBitaui}
    (u-v)\left[b_i(u), b_{\tau i}(v)\right]-&\frac{c_{i,\tau i}}{2} \hbar\left\{b_i(u), b_{\tau i}(v)\right\}-\hbar\left(\left[b_{i, 0}, b_{\tau i}(v)\right]-\left[b_i(u), b_{{\tau i}, 0}\right]\right) \\\notag&\qquad =- \hbar\left(\frac{2 u}{u+v} h_i^{\circ}(u)+\frac{2 v}{u+v} h_{i}^{\circ}(-v)\right).
\end{align}

It is direct to compute
\begin{align*}
    B_i(u)B_{\tau i}(v)
    =\,&\sum_{r=1}^{v_i}\frac{1}{-u-x_{i,r}-\tfrac{\hbar}{2}} \frac{1}{-v+x_{i,r}+\tfrac{\hbar}{2}} y_{i,r}y_{\tau i,r}\\
    &+\sum_{1\leq r\neq s\leq v_i}\frac{1}{-u-x_{i,r}-\tfrac{\hbar}{2}} \frac{1}{-v+x_{i,s}-\tfrac{\hbar}{2}} y_{i,r}y_{\tau i,s},
\end{align*}
and 
\begin{align*}
    B_{\tau i}(v)B_i(u) 
    =\,&\sum_{r=1}^{v_i}\frac{1}{-u-x_{i,r}+\tfrac{\hbar}{2}} \frac{1}{-v+x_{i,r}-\tfrac{\hbar}{2}} y_{\tau i,r}y_{i,r}\\
    &+\sum_{1\leq r\neq s\leq v_i}\frac{1}{-u-x_{i,r}-\tfrac{\hbar}{2}} \frac{1}{-v+x_{i,s}-\tfrac{\hbar}{2}} y_{\tau i,s}y_{i,r}.
\end{align*}
By definition, $B_{i,0}B_{\tau i}(v)$ and $B_{\tau i}(v)B_{i,0}$ (resp. $B_i(u)B_{\tau i,0}$ and $B_{\tau i,0}B_i(u)$) are the coefficients of $u^{-1}$ (resp. $v^{-1}$) in $B_i(u)B_{\tau i}(v)$ and $B_{\tau i}(v)B_i(u)$, respectively.

Hence, the $r\neq s$ terms contribution to the left hand side of \eqref{equ:BBitaui} is zero because of the following identity (notice that $c_{i,\tau i}=0$ or $-1$)
\begin{align*}
    &(u-v-\frac{c_{i,\tau i}}{2}\hbar)\frac{1}{-u-x_{i,r}-\tfrac{\hbar}{2}} \frac{1}{-v+x_{i,s}-\tfrac{\hbar}{2}}\\
    &-(u-v+\frac{c_{i,\tau i}}{2}\hbar)\frac{1}{-u-x_{i,r}-\tfrac{\hbar}{2}} \frac{1}{-v+x_{i,s}-\tfrac{\hbar}{2}}\Big(\frac{x_{i,r}-\tfrac{\hbar}{2}+x_{i,s}}{x_{i,r}+\tfrac{\hbar}{2}+x_{i,s}}\Big)^{-c_{i,\tau i}}\\
    &+\frac{1}{-v+x_{i,s}-\tfrac{\hbar}{2}}-\frac{1}{-v+x_{i,s}-\tfrac{\hbar}{2}}\Big(\frac{x_{i,r}-\tfrac{\hbar}{2}+x_{i,s}}{x_{i,r}+\tfrac{\hbar}{2}+x_{i,s}}\Big)^{-c_{i,\tau i}} \\
    &-\frac{1}{-u-x_{i,r}-\tfrac{\hbar}{2}}+\frac{1}{-u-x_{i,r}-\tfrac{\hbar}{2}}\Big(\frac{x_{i,r}-\tfrac{\hbar}{2}+x_{i,s}}{x_{i,r}+\tfrac{\hbar}{2}+x_{i,s}}\Big)^{-c_{i,\tau i}} =0. 
\end{align*}
Here we have used \eqref{equ:ytauiyi}.

Therefore, the left-hand side of \eqref{equ:BBitaui} is
\begin{align*}
    &(u-v-\frac{c_{i,\tau i}}{2}\hbar)\sum_{r=1}^{v_i}\frac{1}{-u-x_{i,r}-\tfrac{\hbar}{2}} \frac{1}{-v+x_{i,r}+\tfrac{\hbar}{2}} y_{i,r}y_{\tau i,r}\\
    &-(u-v+\frac{c_{i,\tau i}}{2}\hbar)\sum_{r=1}^{v_i}\frac{1}{-u-x_{i,r}+\tfrac{\hbar}{2}} \frac{1}{-v+x_{i,r}-\tfrac{\hbar}{2}} y_{\tau i,r}y_{i,r}\\
    &+\sum_{r=1}^{v_i}\frac{1}{-v+x_{i,r}+\tfrac{\hbar}{2}} y_{i,r}y_{\tau i,r}-\sum_{r=1}^{v_i}\frac{1}{-v+x_{i,r}-\tfrac{\hbar}{2}} y_{\tau i,r}y_{i,r}\\
    &-\sum_{r=1}^{v_i}\frac{1}{-u-x_{i,r}-\tfrac{\hbar}{2}} y_{i,r}y_{\tau i,r}+\sum_{r=1}^{v_i}\frac{1}{-u-x_{i,r}+\tfrac{\hbar}{2}} y_{\tau i,r}y_{i,r}\\
    =&-\frac{1}{u+v}\sum_{r=1}^{v_i}\bigg(\frac{2x_{i,r}+\hbar+\frac{c_{i,\tau i}}{2}\hbar}{u+x_{i,r}+\tfrac{\hbar}{2}}y_{i,r}y_{\tau i, r}-\frac{2x_{i,r}-\hbar-\frac{c_{i,\tau i}}{2}\hbar}{u+x_{i,r}-\tfrac{\hbar}{2}}y_{\tau i,r}y_{i, r}\bigg)\\
    &-\frac{1}{u+v}\sum_{r=1}^{v_i}\bigg(-\frac{2x_{i,r}+\hbar+\frac{c_{i,\tau i}}{2}\hbar}{-v+x_{i,r}+\tfrac{\hbar}{2}}y_{i,r}y_{\tau i, r}+\frac{2x_{i,r}-\hbar-\frac{c_{i,\tau i}}{2}\hbar}{-v+x_{i,r}-\tfrac{\hbar}{2}}y_{\tau i,r}y_{i, r}\bigg)\\
    =&-\hbar\left(\frac{2 u}{u+v} H_i^{\circ}(u)+\frac{2 v}{u+v} H_{i}^{\circ}(-v)\right).
\end{align*}
Here the last equality follows from Corollary \ref{cor:uH}. This finishes the proof of \eqref{equ:BBitaui}.

\subsection{Relation \eqref{eq:commSerre(z)}}\label{sec:BBijc=0}
In this paragraph, we check the following relation 
\[(u+v)\left[b_i(u), b_j(v)\right]=\delta_{\tau i, j}\hbar\left(h_j^\circ(v)-h_i^\circ(u)\right), \qquad c_{ij}=0.\]
There are two cases depending on whether $j$ equals $\tau i$ or not.

\subsubsection{The case $j\neq \tau i$.} 
Since $c_{ij}=0$, $j\neq i$ and there is no edge connecting $i$ and $j$. If further $c_{\tau i,j}=0$, it is obvious to see that 
\[[B_i(u),B_j(v)]=0.\]
Now let us assume $c_{\tau i, j}=-1$. By symmetry, we can assume that the edge connecting $i$ and $\tau j$ has source $i$. Hence, there is also an edge going from $j$ to $\tau i$. Then the commutativity of $B_i(u)$ and $B_j(v)$ follows from the following identity
\begin{align*}
    &V_{\tau j}(x_{i,r}+\tfrac{\hbar}{2})d_{i,r}V_{\tau i}(x_{j,s}+\tfrac{\hbar}{2})d_{j,s}\\
    =\,&(x_{j,s}+\tfrac{\hbar}{2}+x_{i,r})(x_{j,s}+\tfrac{3\hbar}{2}+x_{i,r})V_{\tau j,s}(x_{i,r}+\tfrac{\hbar}{2})V_{\tau i,r}(x_{j,s}+\tfrac{\hbar}{2})d_{i,r}d_{j,s}\\
    =\,&V_{\tau i}(x_{j,s}+\tfrac{\hbar}{2})d_{j,s}V_{\tau j}(x_{i,r}+\tfrac{\hbar}{2})d_{i,r}.
\end{align*}

\subsubsection{The case $j=\tau i$.}
Since $c_{i,\tau i}=0$, \eqref{equ:ytauiyi} shows $[y_{i,r}, y_{\tau i,s}]=0$.
Therefore,
\begin{align*}
(u+v)[B_i(u),B_{\tau i}(v)] 
    =\,&\bigg(\sum_{r=1}^{v_i} \frac{1}{-v+x_{i,r}-\tfrac{\hbar}{2}}y_{\tau i,r}y_{i,r}-\sum_{r=1}^{v_i}\frac{1}{-v+x_{i,r}+\tfrac{\hbar}{2}}y_{i,r}y_{\tau i,r}\bigg)\\
    &-\bigg(\sum_{r=1}^{v_i} \frac{1}{u+x_{i,r}-\tfrac{\hbar}{2}}y_{\tau i,r}y_{i,r}-\sum_{r=1}^{v_i}\frac{1}{u+x_{i,r}+\tfrac{\hbar}{2}}y_{i,r}y_{\tau i,r}\bigg)\\
    =\,&\hbar\big(H_i^\circ(-v)-H^\circ_i(u)\big)
    =\hbar \big(H^\circ_{\tau i}(v)-H^\circ_i(u)\big).
\end{align*}
Here the last equality follows from Corollary \ref{cor:cequal0H}. This finishes the proof of relation \eqref{eq:commSerre(z)}.

\section{Relations (II)}\label{sec:iserre}

\subsection{Generating function formulation of \cref{eq:iSerre}}
\label{sec:gftoiSerre}

In the next subsection, we verify that \cref{eq:iSerre} is preserved under \cref{thm:main}. Our goal in this subsection is to reduce the proof to a simpler case.  Following the approach of \cite{Levendorskii,LZ24}, we show that it is enough to check 
\cref{eq:iSerre} in the special case \( k_1 = k_2 = r = 0 \), assuming that \cref{eq:hcomm}--\cref{eq:bcomm} are already preserved.

The Bernoulli polynomials $\{\Ber_n(x)\}_{n\geq 0}$ are defined to be the coefficients of the following exponential generating function
$$\frac{te^{tx}}{e^t-1}
=\frac{1+tx+\frac{(tx)^2}{2!}+\cdots}{
1+\frac{t}{2!}+\frac{t^2}{3!}+\cdots}
=\sum_{n\geq 0} \Ber_n(x)\frac{t^n}{n!}\in 
\mathbb{Q}[x][[t]].$$
By replacing $x$ by $x+1$, it is easy to see the following property 
$$\Ber_n(x+1)-\Ber_n(x)=nx^{n-1}. $$
Let us define 
$$\tilde{H}_{i,n}
=(-1)^{n+1}\frac{\hbar^{n}}{n+1}\sum_{r=1}^{v_i} \Ber_{n+1}(
\tfrac{1}{\hbar}x_{i,r}+\tfrac{1}{2}) \in 
\Diff_\hbar(T_V^\tau).$$
Recall that $B_{i,s}$ is the coefficient of $z^{-s-1}$ of $B_i(z)$ from \cref{equ:Biz}:
\[B_{i,s}=-\sum_{r=1}^{v_i}(-x_{i,r}-\tfrac{\hbar}{2})^s \frac{\prod_{h\in Q_1\atop s(h)=i}V_{t(h)}(x_{i,r}+\tfrac{\hbar}{2})}{\prod_{h\in Q_1^\tau\atop s(h)=i}(2x_{i,r}+\tfrac{\hbar}{2})}\frac{W_{\tau i}(-x_{i,r}-\tfrac{\hbar}{2})}{V_{i,r}(x_{i,r})}d_{i,r}.\]

\begin{lem}We have 
\begin{equation}
\label{tildeH}
[\tilde{H}_{i,n},B_{i,s}] = B_{i,n+s},\qquad 
[\tilde{H}_{i,n},B_{\tau i,s}] =-(-1)^nB_{\tau i,n+s}.
\end{equation}
\end{lem}
\begin{proof}We can compute
\begin{align*}
[\tilde{H}_{i,n},d_{i,r}]
& 
=(-1)^{n+1}\frac{\hbar^{n}}{n+1}
\left(\Ber_{n+1}(\tfrac{1}{\hbar}x_{i,r}+\tfrac{1}{2})
-\Ber_{n+1}(\tfrac{1}{\hbar}x_{i,r}+\tfrac{3}{2})\right)
d_{i,r}\\
& = 
(-1)^{n}\hbar^{n}
(\tfrac{1}{\hbar}x_{i,r}+\tfrac{1}{2})^{n} d_{i,r}
=(-x_{i,r}-\tfrac{\hbar}{2})^n d_{i,r}.
\end{align*}
Hence, 
$[\tilde{H}_{i,n},B_{i,s}] = B_{i,n+s}$. 
Similarly,
\begin{align*}
[\tilde{H}_{i,n},d_{i,r}^{-1}]
& 
=(-1)^{n+1}\frac{\hbar^{n}}{n+1}
\left(\Ber_{n+1}(\tfrac{1}{\hbar}x_{i,r}+\tfrac{1}{2})
-\Ber_{n+1}(\tfrac{1}{\hbar}x_{i,r}-\tfrac{1}{2})\right)
d_{i,r}^{-1}\\
& = 
(-1)^{n+1}\hbar^{n}
(\tfrac{1}{\hbar}x_{i,r}-\tfrac{1}{2})^{n} d_{i,r}^{-1}
=(-1)^{n+1}(x_{i,r}-\tfrac{\hbar}{2})^n d_{i,r}^{-1}.
\end{align*}
Hence, 
$[\tilde{H}_{i,n},B_{\tau i,s}] = -(-1)^{n}B_{\tau i,n+s}$. 
\end{proof}

\begin{rem}
    Similar elements also appeared in \cite{NakdAHA}. The elements \(\tilde{H}_{i,r}\) are analogues of \(\tilde{h}_{i,r}\) from \cite[Section~3.4]{LZ24}. In contrast to \cite[Section~3.4]{LZ24}, we do not assume that the Cartan matrix \( C = (c_{ij}) \) is invertible.  Therefore, we cannot directly define \(\tilde{h}_{i,r}\) for the shifted twisted Yangians.
\end{rem}

Now we fix $i$ and suppose $c_{i,\tau i}=-1$ (which also implies $i\neq \tau i$). We note that, in \cref{thm:main}, we have 
\[b_{i,s}\mapsto \hbar^{-1}B_{i,s},\qquad h_{i,r}\mapsto \hbar^{-1}H_{i,r}.\]
In what follows, we show that the preservation of the relation \cref{eq:iSerre} can be deduced from its base case with \( k_1 = k_2 = r = 0 \), i.e. 
\begin{equation}
    \label{finiteSerre}
    \hbar^{-1}[B_{i,0},[B_{i,0},B_{\tau i,0}]]=\frac{4}{3}\sum_{p=0}^{\infty}3^{-p}[B_{i,p},H_{\tau i,-p}]
\end{equation}
We remark that the right hand side of \cref{finiteSerre} is a finite sum by the same assumption as in \cref{eq:vanishing}.

\begin{prop}
\label{prop:finite=>affine}
Suppose that \eqref{eq:hcomm}--\eqref{eq:bcomm} are preserved under \cref{gklo}, then \eqref{eq:iSerre} for arbitrary $k_1,k_2,r\geq 0$ follows from its special case \cref{finiteSerre} (where $k_1=k_2=r=0$). 
\end{prop}

\begin{proof}
The argument is essentially the same as that of \cite[Proposition 3.12]{LZ24}, but we repeat here for completeness.

We set 
\[(k_1,k_2\mid r):=\hbar^{-1}\operatorname{Sym}_{k_1,k_2}[B_{i,k_1},[B_{i,k_2},B_{\tau i,r}]].\]
Then we have 
\begin{align}
\label{induc}
&   (k_1+1,k_2\mid r)+(k_1,k_2+1\mid r)-2(k_1,k_2\mid r+1)\\\notag
& \qquad \qquad \qquad =-4\operatorname{Sym}_{k_1,k_2}(-1)^{k_1}[B_{i,k_2},H_{\tau i,k_1+r+1}].
\end{align}
In fact, 
\begin{align*}
    & (k_1+1,k_2\mid r)+(k_1,k_2+1\mid r)-2(k_1,k_2\mid r+1)\\
    =\hbar^{-1}\,&\operatorname{Sym}_{k_1,k_2}
\Big([B_{i,k_1+1},[B_{i,k_2},B_{\tau i,r}]]
    + [B_{i,k_1},[B_{i,k_2+1},B_{\tau i,r}]]
    -2 [B_{i,k_1},[B_{i,k_2},B_{\tau i, r+1}]]
    \Big)\\
    =\hbar^{-1}\,&\operatorname{Sym}_{k_1,k_2}
\Big(
[[B_{i,k_1+1},B_{i,k_2}],B_{\tau i,r}]
+ 2[B_{i,k_1},[B_{i,k_2+1},B_{\tau i,r}]]
-2 [B_{i,k_1},[B_{i,k_2},B_{\tau i, r+1}]]\Big)\\
=\,&
[\{B_{i,k_1},B_{i,k_2}\},B_{\tau i,r}]
+ 
\operatorname{Sym}_{k_1,k_2}\Big(2
\big[B_{i,k_1},-\tfrac{1}{2}
\{B_{i,k_2},B_{\tau i,r}\}-2(-1)^{k_2}
[B_{i,k_1},H_{\tau i,k_2+r+1}]\big]\Big)\\
=\,&-4\operatorname{Sym}_{k_1,k_2}(-1)^{k_2}[B_{i,k_1},H_{\tau i,k_2+r+1}]. 
\end{align*}    
where the second to the last equality follows from \cref{eq:bcomm}.

We first show \cref{eq:iSerre} for the case $(0,0\ |\ r)$ by an induction on $r$. The base case when $r=0$ is just \cref{finiteSerre}. Now we suppose \cref{eq:iSerre} holds for $(0,0\mid r)$, that is
    \begin{equation}\label{00r}
   \hbar^{-1} \left[B_{i,0},\left[B_{i, 0}, B_{\tau i, r}\right]\right]=\frac{4}{3}  \sum_{p=0}^{\infty} 3^{-p}\left[B_{i, p}, H_{\tau i, r-p}\right].
\end{equation}
Bracketing \cref{00r} with $\tilde{H}_{i,k_1}$, we find that
\begin{align*}
    &[\tilde H_{i,k_1},\left[B_{i,0},\left[B_{i, 0}, B_{\tau i, r}\right]\right]]=[[\tilde H_{i,k_1},B_{i,0}],\left[B_{i, 0}, B_{\tau i, r}\right]]+[B_{i,0},[\tilde H_{i,k_1},\left[B_{i, 0}, B_{\tau i, r}\right]]] \\
    &\qquad 
    =\left[B_{i,k_1},\left[B_{i, 0}, B_{\tau i, r}\right]\right]+\left[B_{i,0},\left[B_{i, k_1}, B_{\tau i, r}\right]\right]-(-1)^{k_1}\left[B_{i,0},\left[B_{i, 0}, B_{\tau i, r+k_1}\right]\right],
\end{align*}
where the second equality follows from \cref{tildeH}. Moreover,
\begin{align*}&
     [\tilde H_{i,k_1},\left[B_{i, p}, H_{\tau i, r-p}\right]]=[[\tilde H_{i,k_1},B_{i,p}],H_{\tau i, r-p}]=[B_{i,k_1+p},H_{\tau i, r-p}],
\end{align*}
where the first equality follows from \cref{eq:hcomm} while the second one follows from \cref{tildeH}.
Thus  we conclude that
\begin{equation}
\label{sum1}
    2(k_1,0\ |\ r)-(-1)^{k_1}(0,0\ |\ r+k_1)=\frac{8}{3}\sum_{p=0}^{\infty} 3^{-p}\left[B_{i, k_1+p}, H_{\tau i, r-p}\right].
\end{equation}
On the other hand, by setting $k_1=k_2=0$ in \cref{induc}, we obtain 
\begin{equation}
    \label{sum2}
    2(1,0\mid r)-2(0,0\mid r+1)=-8 [B_{i,0},H_{\tau i,r+1}].
\end{equation}
Solving $(0,0\mid r+1)$ from \cref{sum1,sum2}, we conclude that \cref{eq:iSerre} also holds for the case $(0,0\mid r+1)$. Then by \cref{eq:iSerre} for the case $(0,0\mid r+k_1)$ and \cref{sum1}, we obtain \cref{eq:iSerre} for the case $(k_1,0\mid r)$.

Bracketing $(k_1,0\mid 0)$ with $\tilde H_{i,k_2}$, we get
\begin{align*}
    &(k_1+k_2,0\mid 0)+(k_1,k_2\mid 0)-(-1)^{k_2}(k_1,0\mid k_2)\\
    =\,&\frac{4}{3}(-1)^{k_1}\sum_{p=0}^{\infty}3^{-p}[B_{i,p+k_2},H_{\tau i,k_1-p}]+\frac{4}{3}\sum_{p=0}^{\infty}3^{-p}[B_{i,p+k_1+k_2},H_{\tau i,-p}].
\end{align*}
Since we already obtain \cref{eq:iSerre} for $(k_1+k_2,0\mid 0)$ and $(k_1,0\mid k_2)$, we deduce \cref{eq:iSerre} for $(k_1,k_2\mid 0)$. Finally, bracketing \cref{eq:iSerre} for $(k_1,k_2\mid 0)$ with $\tilde{H}_{i,r}$, we obtain
\begin{align*}
    &(k_1+r,k_2\mid 0)+(k_1,k_2+r\mid 0)-(-1)^{r}(k_1,k_2\mid r)\\
    =\,&\frac{4}{3}(-1)^{k_1}\sum_{p=0}^{\infty}3^{-p}[B_{i,p+k_2+r},H_{\tau i,k_1-p}]+\frac{4}{3}(-1)^{k_2}\sum_{p=0}^{\infty}3^{-p}[B_{i,p+k_1+r},H_{\tau i,k_2-p}].
\end{align*}
By substituting \cref{eq:iSerre} for $(k_1+r,k_2\mid 0)$ and $(k_1,k_2+r\mid 0)$, we obtain $(k_1,k_2\mid r)$.
\end{proof}

\subsection{Serre Relations}
\label{iSerre(z)}
The Serre Relation \eqref{eq:usualSerre(z)} can be checked exactly the same as in \cite[Appendix B.6]{BFN19}, so we omit it.

Now let us check the $\imath$Serre Relation \eqref{eq:iSerre}. 
By Proposition \ref{prop:finite=>affine}, it suffices to check Relation \eqref{eq:iSerre} at $k_1=k_2=0$, which admits a generating function 
\begin{equation}\label{equ:iSerreitaui}
	\hbar^2[b_{i,0},[b_{i,0},b_{\tau(i)}(v)]]=\left(4v
	\hbar [b_i(3v), h_{\tau(i)}(v)]\right)^\circ.
\end{equation}
Recall that $\hbar b_{i,0}$ is sent to $B_{i,0}=-\sum_{r=1}^{v_i}y_{i,r}$.
By direct computation, we get
\begin{align*}
B_{i,0}B_{i,0}B_{\tau i}(v)
	=\,& \sum_{r=1}^{v_i}\frac{1}{-v+x_{i,r}+\tfrac{3\hbar}{2}}y_{i,r}y_{i,r}y_{\tau i,r} 
    +\sum_{1\leq r\neq t\leq v_i}\frac{1}{-v+x_{i,t}-\tfrac{\hbar}{2}}y_{i,r}y_{i,r}y_{\tau i,t}
    \\
	&+\sum_{1\leq r\neq s\leq v_i}\frac{1}{-v+x_{i,s}+\tfrac{\hbar}{2}}y_{i,r}y_{i,s}y_{\tau i,s}
    +\sum_{1\leq r\neq s\leq v_i}\frac{1}{-v+x_{i,r}+\tfrac{\hbar}{2}}y_{i,r}y_{i,s}y_{\tau i,r}
    \\
	&+\sum_{r,s,t\text{ distinct}}
    \frac{1}{-v+x_{i,t}-\tfrac{\hbar}{2}}y_{i,r}y_{i,s}y_{\tau i,t},
\end{align*}
\begin{align*}
B_{i,0}B_{\tau i}(v)B_{i,0}
	=\,& \sum_{r=1}^{v_i}\frac{1}{-v+x_{i,r}+\tfrac{\hbar}{2}}y_{i,r}y_{\tau i,r}y_{i,r}
    +\sum_{1\leq r\neq t\leq v_i}\frac{1}{-v+x_{i,t}-\tfrac{\hbar}{2}}y_{i,r}y_{\tau i,t}y_{i,r}
    \\
    &+\sum_{1\leq r\neq s\leq v_i}\frac{1}{-v+x_{i,s}-\tfrac{\hbar}{2}}y_{i,r}y_{\tau i,s}y_{i,s}
    +\sum_{1\leq r\neq s\leq v_i}\frac{1}{-v+x_{i,r}+\tfrac{\hbar}{2}}y_{i,r}y_{\tau i,r}y_{i,s}
    \\
	&+\sum_{r,s,t\text{ distinct}}\frac{1}{-v+x_{i,t}-\tfrac{\hbar}{2}}y_{i,r}y_{\tau i,t}y_{i,s},
    \text{\qquad and}
\end{align*}
\begin{align*}
B_{\tau i}(v)B_{i,0}B_{i,0}
	=&\sum_{r=1}^{v_i}\frac{1}{-v+x_{i,r}-\tfrac{\hbar}{2}}y_{\tau i,r}y_{i,r}y_{i,r}
    +
    \sum_{1\leq r\neq t\leq v_i}\frac{1}{-v+x_{i,t}-\tfrac{\hbar}{2}}y_{\tau i,t}y_{i,r}y_{i,r}
    \\
	&+\sum_{1\leq r\neq s\leq v_i}\frac{1}{-v+x_{i,s}-\tfrac{\hbar}{2}}y_{\tau i,s}y_{i,r}y_{i,s}
    +
    \sum_{1\leq r\neq s\leq v_i}\frac{1}{-v+x_{i,r}-\tfrac{\hbar}{2}}y_{\tau i,r}y_{i,r}y_{i,s}
    \\
	&+\sum_{r,s,t\text{ distinct}}\frac{1}{-v+x_{i,t}-\tfrac{\hbar}{2}}y_{\tau i,t}y_{i,r}y_{i,s}.
\end{align*}

Notice that 
\[[b_{i,0},[b_{i,0},b_{\tau(i)}(v)]]=b_{i,0}b_{i,0}b_{\tau(i)}(v)-2b_{i,0}b_{\tau(i)}(v)b_{i,0}+b_{\tau(i)}(v)b_{i,0}b_{i,0}.\]
Therefore, the contribution of the terms when $r,s,t$ are distinct to \eqref{equ:iSerreitaui} is 
\begin{align*}
	&\sum_{r,s,t\text{ distinct}}\frac{1}{-v+x_{i,t}-\tfrac{\hbar}{2}}(y_{i,r}y_{i,s}y_{\tau i,t}-2y_{i,r}y_{\tau i,t}y_{i,s}+y_{\tau i,t}y_{i,r}y_{i,s})\\
	=\,&\sum_{r,s,t\text{ distinct}}
    \frac{1}{-v+x_{i,t}-\tfrac{\hbar}{2}}\frac{\hbar(\hbar+x_{i,r}-x_{i,s})}{(x_{i,r}+\tfrac{\hbar}{2}+x_{i,t})(x_{i,s}+\tfrac{\hbar}{2}+x_{i,t})}y_{i,r}y_{i,s}y_{\tau i,t}\\
	=\,&\frac{1}{2}\sum_{r,s,t\text{ distinct}}\frac{1}{-v+x_{i,t}-\tfrac{\hbar}{2}}\frac{\hbar}{(x_{i,r}+\tfrac{\hbar}{2}+x_{i,t})(x_{i,s}+\tfrac{\hbar}{2}+x_{i,t})}\\
	&\cdot \bigg((\hbar+x_{i,r}-x_{i,s})y_{i,r}y_{i,s}y_{\tau i,t}+(\hbar+x_{i,s}-x_{i,r})y_{i,s}y_{i,r}y_{\tau i,t}\bigg) = 0.
\end{align*}
Here the second equality follows from \eqref{equ:ytauiyi} and the last one follows from \eqref{equ:yiyi}.

Similarly, by \eqref{equ:ytauiyi}, the contribution of the terms when $r=s\neq t$ is $0$ because of the following identity
\begin{align*}
	&y_{i,r}y_{i,r}y_{\tau i,t}-2y_{i,r}y_{\tau i,t}y_{i,r}+y_{\tau i,t}y_{i,r}y_{i,r}\\
	=\,&\bigg(1-2\frac{x_{i,r}+\tfrac{\hbar}{2}+x_{i,t}}{x_{i,r}+\tfrac{3\hbar}{2}+x_{i,t}}+\frac{x_{i,r}-\tfrac{\hbar}{2}+x_{i,t}}{x_{i,r}+\tfrac{\hbar}{2}+x_{i,t}}\frac{x_{i,r}+\tfrac{\hbar}{2}+x_{i,t}}{x_{i,r}+\tfrac{3\hbar}{2}+x_{i,t}}\bigg)y_{i,r}y_{i,r}y_{\tau i,t} 
    =0.
\end{align*}

Therefore, by Equations \eqref{equ:yiyi} and \eqref{equ:ytauiyi}, we get
\begin{align}\label{equ:bbbiitaui}
	&[B_{i,0},[B_{i,0},B_{\tau(i)}(v)]]\\
	=\,&\sum_{r=1}^{v_i}\bigg(-\frac{1}{v-x_{i,r}-\tfrac{3\hbar}{2}}y_{i,r}y_{i,r}y_{\tau i,r}+\frac{2}{v-x_{i,r}-\tfrac{\hbar}{2}}y_{i,r}y_{\tau i,r}y_{i,r}-\frac{1}{v-x_{i,r}+\tfrac{\hbar}{2}}y_{\tau i,r}y_{i,r}y_{i,r}\bigg)\notag\\
		&+\sum_{1\leq r\neq s\leq v_i}\frac{1}{v-x_{i,s}+\tfrac{\hbar}{2}}\frac{\hbar(\hbar-4x_{i,s})}{(x_{i,r}-x_{i,s})(x_{i,r}+x_{i,s}+\tfrac{\hbar}{2})}y_{i,r}y_{\tau i,s}y_{i,s}\notag\\
	&+\sum_{1\leq r\neq s\leq v_i}\frac{1}{v-x_{i,s}-\tfrac{\hbar}{2}}\frac{\hbar(\hbar+4x_{i,s})}{(x_{i,r}-x_{i,s}-\hbar)(\tfrac{3\hbar}{2}+x_{i,r}+x_{i,s})}y_{i,r}y_{i,s}y_{\tau i,s}.\notag
\end{align}

Since
$$H_i(u) := \frac{(-1)^{v_i-1+\delta_{i\to \tau i}}}{2u}\frac{W_i(-u)W_{\tau i}(u)}{V_i(-u+\frac{\hbar}{2})V_i(-u-\tfrac{\hbar}{2})}\prod_{h\in Q_1\atop s(h)=i}V_{t(h)}(-u)\prod_{h\in Q_1\atop s(h)=\tau i}V_{t(h)}(u),$$
we get
\begin{align*}
4v[B_i(3v),H_i(-v)]
	=\,&2\hbar(-1)^{v_i}\frac{W_i(v)W_{\tau i}(-v)}{V_i(v+\frac{\hbar}{2})V_i(v-\tfrac{\hbar}{2})}\prod_{h\notin Q_1^\tau \atop s(h)=i}V_{t(h)}(v)\prod_{h\notin Q_1^\tau\atop s(h)=\tau i}V_{t(h)}(-v)\\
	&\cdot\sum_{r=1}^{v_i}\frac{1}{v-\tfrac{3\hbar}{2}-x_{i,r}} \prod_{h\in Q_1^\tau \atop s(h)=i}V_{\tau i,r}(v) \prod_{h\in Q_1^\tau \atop s(h)=\tau i}V_{i,r}(-v)y_{i,r}.
\end{align*}
Hence, the function $4v[B_i(3v),H_i(-v)]$ has simple poles at \[\bigg\{x_{i,r}+\tfrac{3\hbar}{2}, x_{i,r}-\tfrac{\hbar}{2}, x_{i,r}+\tfrac{\hbar}{2} \mid 1\leq r\leq v_i \bigg\}.\] 
Applying Lemma \ref{lem:partialfrac} to $4v[B_i(3v),H_i(-v)]$, we get  
\begin{align*}
	&(4v\hbar[B_i(3v),H_i(-v)])^\circ\\
	=\,&\sum_{r=1}^{v_i}\frac{(-1)^{v_i}}{v-\tfrac{3\hbar}{2}-x_{i,r}}\frac{W_i(x_{i,r}+\tfrac{3\hbar}{2} )W_{\tau i}(-x_{i,r}-\tfrac{3\hbar}{2} )}{V_{i,r}(x_{i,r}+2\hbar )V_{i,r}(x_{i,r}+\hbar )}\prod_{h\notin Q_1^\tau \atop s(h)=i}V_{t(h)}(x_{i,r}+\tfrac{3\hbar}{2} )\prod_{h\notin Q_1^\tau\atop s(h)=\tau i}V_{t(h)}(-x_{i,r}-\tfrac{3\hbar}{2} )\\
	&\cdot \prod_{h\in Q_1^\tau \atop s(h)=i}V_{\tau i,r}(x_{i,r}+\tfrac{3\hbar}{2} ) \prod_{h\in Q_1^\tau \atop s(h)=\tau i}V_{i,r}(-x_{i,r}-\tfrac{3\hbar}{2} )y_{i,r}\\
	&+\sum_{r=1}^{v_i}\frac{(-1)^{v_i}}{v-x_{i,r}+\frac{\hbar}{2}}\frac{W_i(x_{i,r}-\tfrac{\hbar}{2})W_{\tau i}(-x_{i,r}+\tfrac{\hbar}{2})}{V_{i,r}(x_{i,r})V_{i,r}(x_{i,r}-\hbar)}\prod_{h\notin Q_1^\tau \atop s(h)=i}V_{t(h)}(x_{i,r}-\tfrac{\hbar}{2})\prod_{h\notin Q_1^\tau\atop s(h)=\tau i}V_{t(h)}(-x_{i,r}+\tfrac{\hbar}{2})\\  
	&\cdot\prod_{h\in Q_1^\tau \atop s(h)=i}V_{\tau i,r}(x_{i,r}-\tfrac{\hbar}{2}) \prod_{h\in Q_1^\tau \atop s(h)=\tau i}V_{i,r}(-x_{i,r}+\tfrac{\hbar}{2})y_{i,r}\\
	&-\sum_{r=1}^{v_i}\frac{2(-1)^{v_i}}{v-x_{i,r}-\tfrac{\hbar}{2}}\frac{W_i(x_{i,r}+\tfrac{\hbar}{2})W_{\tau i}(-x_{i,r}-\tfrac{\hbar}{2})}{V_{i,r}(x_{i,r}+\hbar)V_{i,r}(x_{i,r})}\prod_{h\notin Q_1^\tau \atop s(h)=i}V_{t(h)}(x_{i,r}+\tfrac{\hbar}{2})\prod_{h\notin Q_1^\tau\atop s(h)=\tau i}V_{t(h)}(-x_{i,r}-\tfrac{\hbar}{2})\\
	&\cdot \prod_{h\in Q_1^\tau \atop s(h)=i}V_{\tau i,r}(x_{i,r}+\tfrac{\hbar}{2}) \prod_{h\in Q_1^\tau \atop s(h)=\tau i}V_{i,r}(-x_{i,r}-\tfrac{\hbar}{2})y_{i,r}\\
	&-\sum_{1\leq r\neq s\leq v_i}\frac{2\hbar(-1)^{v_i}}{v-x_{i,s}+\frac{\hbar}{2}}\frac{W_i(x_{i,s}-\tfrac{\hbar}{2})W_{\tau i}(-x_{i,s}+\tfrac{\hbar}{2})}{V_{i,s}(x_{i,s})V_{i,s}(x_{i,s}-\hbar)}\prod_{h\notin Q_1^\tau \atop s(h)=i}V_{t(h)}(x_{i,s}-\tfrac{\hbar}{2})\prod_{h\notin Q_1^\tau\atop s(h)=\tau i}V_{t(h)}(-x_{i,s}+\tfrac{\hbar}{2})\\
	&\cdot\frac{2x_{i,s}-\tfrac{\hbar}{2}}{(x_{i,s}-2\hbar-x_{i,r})(x_{i,s}-\tfrac{\hbar}{2}+x_{i,r})} \prod_{h\in Q_1^\tau \atop s(h)=i}V_{\tau i,s}(x_{i,s}-\tfrac{\hbar}{2}) \prod_{h\in Q_1^\tau \atop s(h)=\tau i}V_{i,s}(-x_{i,s}+\tfrac{\hbar}{2})y_{i,r}\\
	&+\sum_{1\leq r\neq s\leq v_i}\frac{2\hbar(-1)^{v_i}}{v-x_{i,s}-\tfrac{\hbar}{2}}\frac{W_i(x_{i,s}+\tfrac{\hbar}{2})W_{\tau i}(-x_{i,s}-\tfrac{\hbar}{2})}{V_{i,s}(x_{i,s}+\hbar)V_{i,s}(x_{i,s})}\prod_{h\notin Q_1^\tau \atop s(h)=i}V_{t(h)}(x_{i,s}+\tfrac{\hbar}{2})\prod_{h\notin Q_1^\tau\atop s(h)=\tau i}V_{t(h)}(-x_{i,s}-\tfrac{\hbar}{2})\\
	&\cdot\frac{2x_{i,s}+\tfrac{\hbar}{2}}{(x_{i,s}-\hbar-x_{i,r})(x_{i,s}+\tfrac{\hbar}{2}+x_{i,r})} \prod_{h\in Q_1^\tau \atop s(h)=i}V_{\tau i,s}(x_{i,s}+\tfrac{\hbar}{2}) \prod_{h\in Q_1^\tau \atop s(h)=\tau i}V_{i,s}(-x_{i,s}-\tfrac{\hbar}{2})y_{i,r}.
\end{align*}
which can be shown equal to $[B_{i,0},[B_{i,0},B_{\tau(i)}(v)]]$ from \eqref{equ:bbbiitaui} termwise. This finishes the proof of \eqref{equ:iSerreitaui}.

\bibliographystyle{alpha}
\bibliography{TruSTY.bib}

\end{document}